\documentclass{amsart}
\usepackage{amsmath,amsthm,amssymb,bbm}
\usepackage{graphicx,enumerate, xcolor}
\usepackage[parfill]{parskip}
\usepackage[
colorlinks=true,
linkcolor=black, 
anchorcolor=black,
citecolor=black, 
urlcolor=black, 
]{hyperref}
\usepackage{mathrsfs}
\usepackage[all]{xy}

\newtheorem{theorem}{Theorem}
\newtheorem{lemma}[theorem]{Lemma}
\newtheorem{corollary}[theorem]{Corollary}
\newtheorem{proposition}[theorem]{Proposition}

\theoremstyle{definition}

\newtheorem{example}[theorem]{Example}
\newtheorem{remark}[theorem]{Remark}

\newtheorem{definition}[theorem]{Definition}

\newcommand{\cC}{\mathscr{C}}
\newcommand{\cK}{\mathscr{K}}

\newcommand{\cD}{\mathscr{D}}

\newcommand{\cB}{\mathscr{B}}

\newcommand{\cZ}{\mathscr{Z}}

\newcommand{\csf}{\text{-}\mathrm{csf}}

\newcommand{\lmod}{\text{-}\mathrm{mod}}

\newcommand{\csfcat}{\mathfrak{M}^{\mathrm{dg}}}

\newcommand{\tworep}{\text{-}\mathfrak{mod}^{\mathrm{pre}}}
\newcommand{\dgmod}{\text{-}\mathrm{mod}^{\mathrm{dg}}}

\newcommand*\cocolon{
        \nobreak
        \mskip6mu plus1mu
        \mathpunct{}
        \nonscript
        \mkern-\thinmuskip
        {:}%
        \mskip2mu
        \relax
}

\newcommand{\adj}[4]{#1\colon #2\rightleftarrows #3\cocolon #4}
\newcommand{\shift}[1]{\langle #1\rangle}

\newcommand{\diag}{\operatorname{diag}}
\newcommand{\Hom}{\mathrm{Hom}}
\newcommand{\End}{\mathrm{End}}

\newcommand{\rad}{\operatorname{rad}}

\newcommand{\Ann}{\operatorname{Ann}}
\newcommand{\op}{\mathrm{op}} 

\newcommand{\Spec}{\operatorname{Spec}}
\newcommand{\MaxSpec}{\operatorname{MaxSpec}}

\newcommand{\rF}{\mathrm{F}}
\newcommand{\rG}{\mathrm{G}}
\newcommand{\rH}{\mathrm{H}}

\newcommand{\rX}{\mathrm{X}}
\newcommand{\rY}{\mathrm{Y}}

\newcommand{\A}{\mathcal{A}}

\newcommand{\C}{\mathcal{C}}
\def\D{{\mathcal{D}}}
\newcommand{\E}{\mathcal{E}}

\newcommand{\I}{\mathcal{I}}
\newcommand{\J}{\mathcal{J}}
\newcommand{\K}{\mathcal{K}}
\def\L{{\mathcal{L}}}

\newcommand{\R}{\mathcal{R}}

\newcommand{\Y}{\mathcal{Y}}
\newcommand{\Z}{\mathcal{Z}}

\newcommand{\del}{\partial}
\newcommand{\id}{\mathrm{id}}

\newcommand{\one}{\mathbbm{1}}

\newcommand{\bfM}{\mathbf{M}}
\newcommand{\bfK}{\mathbf{K}}

\newcommand{\bfG}{\mathbf{G}}
\newcommand{\bfP}{\mathbf{P}}

\newcommand{\bfN}{\mathbf{N}}

\newcommand{\bfR}{\mathbf{R}}
\newcommand{\bfC}{\mathbf{C}}

\newcommand{\bfI}{\mathbf{I}}
\newcommand{\bfJ}{\mathbf{J}}

\newcommand{\tone}{\mathtt{1}}
\newcommand{\tzero}{\mathtt{0}}
\newcommand{\ti}{\mathtt{i}}
\newcommand{\tj}{\mathtt{j}}
\newcommand{\tk}{\mathtt{k}}
\newcommand{\tl}{\mathtt{l}}

\newcommand{\tn}{\mathtt{n}}
\newcommand{\tI}{\mathtt{I}}
\newcommand{\tJ}{\mathtt{J}}

\newcommand{\tE}{\mathtt{E}}

\newcommand{\tX}{\mathtt{X}}

\newcommand{\steq}{-}

\newcommand{\ov}[1]{\overline{#1}}
\newcommand{\mat}[1]{\left(\begin{smallmatrix}#1\end{smallmatrix}\right)}

\numberwithin{equation}{subsection}
\numberwithin{theorem}{subsection}

\begin{document}

\title[Dg cell 2-representations]{Differential graded cell 2-representations}
\author{Robert Laugwitz}
\address{School of Mathematical Sciences,
University of Nottingham, University Park, Nottingham, NG7 2RD, UK}
\email{robert.laugwitz@nottingham.ac.uk}

\author{Vanessa Miemietz}
\address{
School of Mathematics, University of East Anglia, Norwich NR4 7TJ, UK}
\email{v.miemietz@uea.ac.uk}
\urladdr{https://www.uea.ac.uk/~byr09xgu/}

\date{January 17, 2025}

\begin{abstract}
This article develops a theory of cell combinatorics and cell $2$-rep\-resentations for differential graded $2$-categories. We introduce two types of partial preorders, called the strong and weak preorder. We then analyse and compare them. The weak preorder is more easily tractable, while the strong preorder is more closely related to the combinatorics of the associated homotopy $2$-representations. To each left cell, we associate a maximal ideal spectrum, and each maximal ideal gives rise to a differential graded cell $2$-representation. We prove that any strong cell is contained in a weak cell and that there is a bijection between the corresponding maximal ideal spectra. Finally, we classify weak and strong cell $2$-representations for dg $2$-categories of projective bimodules over finite-dimensional differential graded algebras.
\end{abstract}


\maketitle

\tableofcontents

\section{Introduction}\label{s0}

\subsection{Motivation}
Traditionally, representation theory studies algebras via the actions of their elements as endomorphisms of vector spaces. Categorification lifts such actions to a higher categorical level, by replacing elements by $1$-morphisms, which now act as functors on categories, and adding an additional layer of $2$-morphisms, acting via natural transformations. One recovers the original structures by passing to Grothendieck groups. Some of the most important examples of categorification in representation theory include Soergel bimodules, categorifying Hecke algebras, and Kac--Moody $2$-categories, categorifying quantum groups. Studying these examples has led to tremendous progress in representation theory, as evidenced by solutions to long-standing conjectures, see e.g. \cite{CR,EW,W}. 

Systematic studies of categorical representations, or $2$-representations, emerged from two angles. One is the study of module categories for fusion or, more generally, abelian tensor categories, as developed by Etingof--Ostrik and others, see \cite{EGNO}. Many $2$-categories arising in categorification in representation theory are, however, not abelian, and passing to their abelianisations sacrifices the existence of adjunctions.  This motivates the second angle, namely the development of a $2$-representation theory encompassing these classes of examples, which was initiated in \cite{MM1} and focuses on the additive structures of the relevant $2$-categories and $2$-representations. This has led to a theory of finitary $2$-representations of finitary $2$-categories, see also \cite{MM2,MM3,MM5,MMMZ,MMMTZ} among others. The definition of a finitary $2$-category imposes strong finiteness conditions, and should be thought of as the $2$-categorical analogue of a finite-dimensional algebra. Relaxations of these conditions, allowing for the treatment of larger examples, were studied in \cite{Mac1,Mac2}.

However, there is emerging evidence that many features of categorification require the use of triangulated structures. Examples include categorified braid groups and categorical braid group actions \cite{ST,KS,Ro1,Ro2, RZ}, where braid relations hold only up to homotopy. Similarly, \cite{MMV} describes triangulated $2$-representations of Soergel bimodules in affine type $A$, which do not arise as triangulated hulls of finitary $2$-representations.   Further examples of categorification on the level of triangulated categories include, e.g., \cite{BFK,CR,CL,JY}. This rich supply of examples motivates the development of a $2$-representation theory allowing $2$-categories to act on homotopy categories. 

A technical obstacle when working directly with triangulated $2$-categories and $2$-rep\-resentations is the lack of functoriality of crucial constructions such as taking cones, which motivates the use of further enriched structures. One categorically appealing choice here is to enrich in $\infty$-categories such as simplicial categories (e.g. the dg nerves of dg categories), but keeping track of higher morphisms is not conducive to doing explicit computations in examples. Thus, our compromise is to enrich by the $1$-category of (pretriangulated) dg categories and to consider pretriangulated $2$-categories and $2$-representations. This approach provides the desired functoriality properties while maintaining concreteness for applications. Indeed, there are examples of categorification using very explicit constructions of dg algebras and bimodules, such as dg $2$-categories coming from quantum group categorification at the complex parameter $i$ (see e.g. \cite{EQ1, EQ2, ElQ, EL, KQ, KT, Ti1}) or contact geometry (see e.g.\ \cite{Ti2, Ti3}).

A systematic treatment of pretriangulated $2$-representations of dg $2$-categories was initiated in \cite{LM2}. In the latter article, we introduce the relevant notions and study cyclic pretriangulated $2$-representations via dg algebra $1$-morphisms, inspired by the analogous theory for tensor categories in \cite{EGNO}. In the present article, we continue to work with the same setup.
 
\subsection{Cell theory in finitary $2$-representation theory}
Classically, Green's cells for semigroups \cite{Gr} and Kazhdan--Lusztig cells for Hecke algebras \cite{KaLu}, combinatorially defined as equivalence classes with respect to certain (left, right and two-sided) partial preorders on (basis) elements, allow for a significant reduction of the problem of constructing simple representations. More precisely, to each simple representation one can associate a unique two-sided cell, its apex. Moreover, simple representations with a fixed apex are in bijection with simple representations associated to the intersection of certain right and left cells (resulting in a so-called $H$-cell) inside this apex. This allows for a vast reduction in the complexity of computations. 

The ideas of classical cell theory were generalised to finitary $2$-representation theory in \cite{MM1}, and have been one of the most important tools for constructing and classifying simple $2$-representations, in the original literature referred to as \emph{simple transitive} $2$-representations, of finitary $2$-categories such as certain quotients of categorified quantum groups (see e.g.\ \cite{MM5, Mac1}) and Soergel bimodules (see e.g.\ \cite{MMMTZ}). In particular, the existence of an apex for a simple $2$-representation and (in the presence of adjunctions) the reduction of the classification of simple $2$-representations to those associated to $H$-cells were proved in \cite{ChMa} and \cite{MMMZ}, respectively.

A crucial role in the development of the theory is played by so-called \emph{cell $2$-representations}. More precisely,
any left cell $\L$, defined as an equivalence class of indecomposable $1$-morphisms with respect to a certain partial preorder,
generates a certain $2$-representation $\bfR_\L$, which turns out to have a unique maximal ideal. The cell $2$-representation is then the quotient of $\bfR_\L$ by this ideal and is simple by construction. In many cases, cell $2$-representations exhaust all simple $2$-representations, see e.g. \cite{MM5, MMZ2}. In other cases, all simple $2$-representations can be constructed from cell $2$-representations, see e.g. \cite{JM, MMZ3}.

\subsection{Cell theory for dg $2$-categories}
In this article, we generalise cell theory to pretriangulated $2$-categories, while at the same time vastly relaxing finiteness assumptions compared to \cite{MM1}. 
We introduce two natural generalisations of the partial preorder given in \cite{MM1} to dg indecomposable $1$-morphisms, which we call the \emph{strong} and \emph{weak} partial preorders, denoting them by the symbols $\leq$ and $\preceq$. This leads to two different notions of (left) cells, and two different $2$-representations generated by a (strong or weak) cell $\L$, denoted by  $\bfR_\L^\leq$ and  $\bfR_\L^\preceq$, respectively. 
The weak order is defined in terms of containment of left (dg) $2$-ideals generated by the identities on the respective dg indecomposable $1$-morphisms. In contrast, the strong order is defined by containment of dg $2$-subrepresentations generated by the respective dg indecomposable $1$-morphisms inside principal (i.e.\ representable) pretriangulated $2$-representations. 

Both notions of cell structure have their advantages.  The weak cells are often easier to determine,  see e.g. Section \ref{secweakCA}, because, informally speaking, weak cells do not see the differential (see Remark \ref{explicit}) and are instead determined by the underlying $\Bbbk$-linear $2$-category. On the other hand, strong cells are closely related to a notion of triangulated cell, which we develop in Section \ref{sectriangcell}. In particular, assuming the $2$-subcategory whose $2$-morphisms consist of all cycles is locally Krull--Schmidt,  (non-acylic) strong and  triangulated left cells coincide, see Corollary \ref{strongequalstriang}.

Moreover, relaxing finiteness assumptions leads to loss of uniqueness of cell $2$-representations associated with a fixed left cell. Instead of unique maximal ideals, we obtain maximal ideal spectra $\MaxSpec(\bfR^\leq_{\L})$, respectively $\MaxSpec(\bfR^\preceq_{\L})$. Each maximal ideal $\bfI$ (in $\MaxSpec(\bfR^\leq_{\L})$, respectively $\MaxSpec(\bfR^\preceq_{\L})$) gives rise to a (strong or weak) dg cell $2$-representation $\bfC_{\bfI}^\leq$, respectively $\bfC_{\bfI}^\preceq$, both of which are quotient-simple, in the sense of not having any non-trivial quotients. This concept of spectrum should be clearly distinguished from the Balmer spectrum for tensor triangulated categories (see e.g.\ \cite{B1, B2}). The latter considers tensor ideals, i.e. ideals generated by identities, while in our theory, it is crucial that ideals do not necessarily need to contain identities. However,
for the simple example of a dg $2$-category with only one object, whose endomorphism category consist of bounded complexes over a (finitely generated) commutative ring, our notion of spectrum recovers the usual algebro-geometric definition of the maximal ideal spectrum of the ring (Example \ref{geomex}).

One of our main results, Theorem \ref{maxspecbij}, compares the maximal ideal spectra obtained via the two constructions, and shows that given a strong left cell $\L$, which is necessarily contained in a weak left cell $\L'$, there is a bijection between $\MaxSpec(\bfR^\leq_{\L})$ and $\MaxSpec(\bfR^\preceq_{\L'})$. This bijection induces a containment of the corresponding strong cell $2$-representation in its weak counterpart (Proposition \ref{cellinclude}).

\subsection{The example of $\cC_A$}
In finitary $2$-representation theory, a prominent example is the $2$-category $\cC_A$ of projective bimodules for a finite-dimensional algebra $A$. This $2$-category should be thought of as a $2$-categorical analogue of a matrix algebra. In particular,  in a certain sense, any finitary $2$-category $\cD$ has a filtration by $2$-categories $\cD_\J$, labelled by two-sided cells $J$, such that the restriction of any simple $2$-representation with apex $\J$ to $\cD_\J$ factors through some $2$-category of the form $\cC_A$. In this sense, (cyclotomic quotients of) categorified Kac--Moody algebras are filtered by $2$-categories of projective bimodules over cyclotomic quiver Hecke algebras \cite{MM5, Mac1}. Crucially, simple $2$-representations of $\cC_A$ are cell $2$-representations and, up to equivalence, there are (for $A$ connected) precisely two of these, the trivial one on $\Bbbk\lmod$ where all projective bimodules act as zero, and the natural one on the category of projective $A$-modules \cite{MM5, MMZ2}.

These $2$-categories of projective bimodules admit a natural generalisation to the world of dg $2$-categories. For a finite-dimensional dg algebra $A$ (with a choice of primitive idempotents annihilated by the differential), we define the dg $2$-category $\cC_A$, whose $1$-morphisms are the thick closure of the identity and projective dg $A$-$A$-bimodules. For this class of examples, we completely classify both weak and strong cell $2$-representations (Theorem \ref{allweakcell2reps}, Corollary \ref{allstrongcell2reps}). In particular, in this case, strong and weak cell $2$-representations are dg equivalent. 

\subsection{Organisation of the article} Section \ref{sec:dggen} contains generalities on dg categories, before providing background on dg $2$-categories and their pretriangulated $2$-representations in Section \ref{dg2}. Section \ref{cellchapter} is the technical core of the article. Here, we define our strong and weak partial preorders in Section \ref{cellcomb} and the respective associated pretriangulated $2$-representations $\bfR_\L^\leq$ and $\bfR_\L^\preceq$. This leads to the definition of the maximal ideal spectra of both in Section \ref{secspectra} and the construction and analysis of cell $2$-representations in Section \ref{cell2section}. We consider dg idempotent completions in Section \ref{secidempcomp} and,  in Section \ref{locendosec}, assuming the existence of a generating set of $1$-morphisms with local endomorphism rings, relate the cell $2$-representations of a dg $2$-category and of its dg idempotent completion. We apply our results to dg $2$-categories obtained from finitary $2$-categories by taking bounded complexes of $1$-morphisms in Section \ref{secfinitarycomp}. Finally, we define a triangulated cell structure and make the connection to the strong cell structure in Section \ref{sectriangcell}. To conclude the article, we discuss the example of $\cC_A$ in detail and classify its weak and strong cell $2$-representations in Section \ref{CAsec}, which ends with a discussion of some explicit examples.

\subsection*{Acknowledgements}

The authors would like to thank Ben Elias, Gustavo Jasso, and Marco Mackaay for helpful conversations related to this work. We further like to thank the referees for helpful comments on the manuscript. 
R.~L. was supported by a Nottingham Research Fellowship and V.~M. was supported by EPSRC grant EP/S017216/1.

\section{Generalities}
\label{sec:dggen}

\subsection{Dg categories}\label{dgsect}

Let $\Bbbk$ be a field and let $C(\Bbbk\lmod)$ denote the symmetric monoidal category of cochain complexes of $\Bbbk$-modules. Its objects are $\mathbb{Z}$-graded $\Bbbk$-vector spaces endowed with a differential $\del$ of degree $+1$ satisfying $\del^2=0$, and whose morphisms preserve the $\mathbb{Z}$-degree and commute with the differentials. 

A category $\C$ enriched over $C(\Bbbk\lmod)$ is called a {\bf dg category} (see e.g.  \cite[Section~2.1]{Ke}, \cite[Section~1.6]{K}). We will always assume that any dg category in this article is small.

We define the dg category $\Bbbk\dgmod$ of {\bf dg $\Bbbk$-modules}  to have the same objects as $C(\Bbbk\lmod)$ but arbitrary $\Bbbk$-linear maps as morphisms. A morphism $f\colon V\to W$ has  degree $n$ if $f(V^k)\subset W^{k+n}$ for each $k\in \mathbb{Z}$. The differential of a morphism of degree $n$ is given by 
$\del(f)=\del_W\circ f-(-1)^n f\circ \del_V.$  This way, $\Bbbk\dgmod$ is enriched over $C(\Bbbk\lmod)$. It has a {\bf shift functor} given by $(V\langle 1\rangle)^k=V^{k+1}$, for $k\in \mathbb{Z}$.

For a dg category $\C$, the category $\Z(\C)$ is the $\Bbbk$-linear category of {\bf dg morphisms} with the same objects as $\C$ and $\Bbbk$-vector spaces
$$\Hom_{\Z(\C)}(X,Y)=\left\{f\in \Hom_{\C}(X,Y)\vert \deg f=0\text{ and }\del f=0 \right\} $$ as morphisms. In particular, $\Z(\Bbbk\dgmod)=C(\Bbbk\lmod)$. We use the terminology of dg idempotent, dg direct summand, dg indecomposable etc. for properties referring to or defined by morphisms in $\Z(\C)$.

If $\C$ is a dg category, we define the dg category $\C\dgmod$ of {\bf dg modules over $\C$} as the dg category of dg functors $\C^\op\to \Bbbk\dgmod$ (cf. \cite[Section 1.2]{Ke}, \cite[Section 2.2]{Or1}). The Yoneda lemma then yields a fully faithful dg functor 
$$\C\to \C\dgmod, \qquad X\mapsto X^\vee:=\Hom_\C(-,X),$$
which we refer to as a {\bf free} dg $\C$-module. Recall that for any $\C$-module $M$, there is a dg isomorphism 
$\Hom_{\C\dgmod}(X^\vee,M)\cong M(X).$

For a dg category $\C$, its {\bf dg idempotent completion} $\C^\circ$ has objects $X_e$ for any dg idempotent $e\in \End_\C(X)$ (meaning  $\deg e=0$ and $\del e=0$)  and morphism spaces
$$\Hom_{\C^\circ}(X_e,Y_f)=f\Hom_{\C}(X,Y)e.$$
Note that $\C^\circ$ is a dg category and the natural embedding $\C\hookrightarrow \C^\circ$ is a dg functor. Moreover, $(\C^\circ)^\circ$ is dg equivalent to $\C^\circ$.

We can extend a dg functor $F\colon \C\to \D$ and a natural transformation of dg functors $\phi\colon F\to G$ to the dg idempotent completions by setting 
\begin{align}\begin{split}
F^\circ \colon \C^\circ \to \D^\circ, \quad F(X_e)=F(X)_{F(e)}, \\ \phi^\circ\colon F^\circ\to G^\circ, \quad \phi^\circ_{X_e}:=G(e)\circ\phi_X\circ F(e).\end{split}\label{functorcirc}
\end{align}

\subsection{Pretriangulated categories}\label{section:pretriang}

We say that a dg category $\C$ is {\bf pretriangulated} if its Yoneda embedding into $\C\dgmod$ via $X\mapsto X^\vee=\Hom_{\C}(-,X)$ is closed under shifts and cones. In the terminology of \cite[Section 2.4]{D}, such $\C$ is called \emph{strongly} pretriangulated. Notice that we do not require $\C$ to be closed under dg direct summands. We call a full dg subcategory $\D\subseteq \C$ a {\bf pretriangulated subcategory} if it is closed under dg isomorphisms, shifts and cones, and call $\D$ {\bf thick} if it is additionally closed under all direct summands which exist in $\C$.

Let $\C$ be a pretriangulated category. 
For a set $\tX$ of objects in $\C$, we denote the {\bf thick closure} of $\tX$ in $\C$ by $\widehat{\tX}$. That is, $\widehat{\tX}$ is the smallest thick pretriangulated subcategory of $\C$ containing $\tX$. An object $X$ in $\C$ is called a {\bf (classical) generator} for $\C$ if the thick closure $\widehat{\lbrace X \rbrace}$ of $X$ is $\C$.

For a dg category $\C$, we write $\ov{\C}$ for the dg category of (one-sided) twisted complexes in $\C$. Explicitly, we let $\overline{\C}$ be the dg category whose
\begin{itemize}
\item objects are pairs $(X=\bigoplus_{m=1}^s X_m, \alpha=(\alpha_{k,l})_{k,l})$ where the $X_m$ are shifts of objects in $\C$ and $\alpha_{k,l}\in \Hom_\C(X_l,X_k)$, $\alpha_{k,l} = 0$ for all $k\geq l$ such that the matrix $\del \cdot \mathtt{I}_s +\left((\alpha_{kl})_*\right)_{kl}$ acts as a differential on $\bigoplus_{m=1}^s X_m^\vee$ in  $\C\dgmod$ (here $\mathtt{I}_s$ denotes the identity matrix);
\item morphisms are matrices of morphisms between the underlying objects, with the differential of a homogeneous morphism
$$\gamma = (\gamma_{n,m})_{n,m} \colon \left(\bigoplus_{m=1}^s X_m, \alpha=(\alpha_{k,l})_{k,l}\right)\longrightarrow \left(\bigoplus_{n=1}^t Y_n, \beta=(\beta_{k,l})_{k,l}\right)$$
given by
$$ \partial\left((\gamma_{n,m})_{n,m} \right):=  (\partial \gamma_{n,m}+(\beta\gamma)_{n,m}-(-1)^{\deg \gamma}(\gamma\alpha)_{n,m})_{n,m}.$$
\end{itemize}

Note that here $\bigoplus_{i=m}^s X_m$ denotes an ordered list of objects rather than a direct sum internal to $\C$. This way, $\ov{\C}$ can be given the explicit additive structure
$$(X,\alpha)\oplus (Y,\beta):=\left(X\oplus Y,\begin{pmatrix}
\alpha&0\\0&\beta
\end{pmatrix}\right),$$
where $X\oplus Y$ is simply the concatenation of ordered lists. This additive structure is strict, in the sense that $(X\oplus Y)\oplus Z=X\oplus (Y\oplus Z)$.

Observe that $\ov{\C}$ is a pretriangulated category, in fact it is the smallest pretriangulated category containing $\C$, see \cite[Section~1]{LO}, \cite[Section 3.2]{AL}, and is dg equivalent to $\C\csf$, the dg category of compact semi-free dg $\C$-modules, cf. \cite[Section 2.3]{Or2}.

For a dg morphism $f\colon X\to Y$ in $\ov{\C}$, with $X=(\bigoplus_{i=1}^tX_i,\alpha)$, $Y=(\bigoplus_{i=1}^tY_i,\beta)$, the {\bf cone} $C_f$ of $f$ is given by 
$$C_f=\Big(Y\oplus X\shift{1}, \mat{\beta& -f\\ 0&\alpha}\Big),$$
and there are natural dg morphisms
$C_f\shift{-1}\to X$ and $Y\to C_f,$
such that pre- (respectively, post-) composition with $f$ yields a null-homotopic morphism.

If $F\colon \C\to \D$ is a dg functor, we obtain an induced dg functor 
\begin{align}\label{functorov}
\ov{F}\colon \ov{\C}\to \ov{D}, \quad \ov{F}\Big(\bigoplus_{i=1}^t X_i,\alpha\Big)=\Big(\bigoplus_{i=1}^t F(X_i),F(\alpha)\Big),
\end{align}
by applying $F$ component-wise to $\alpha$ and to morphisms in $\ov{\C}$. Observe that $\ov{G\circ F}=\ov{G}\circ\ov{F}$ for composable dg functors $F\colon \C\to \D$, $G\colon \D\to \E$. Further, a natural transformation $\tau\colon F\to G$ induces a natural transformation $\ov{\tau}\colon \ov{F}\to \ov{G}$ via the diagonal matrix
\begin{align}
\tau_{(\oplus_{i=1}^t X_i,\alpha)}=\diag(\tau_{X_1},\ldots, \tau_{X_t}).\end{align}
In total, we obtain a dg $2$-functor $\ov{(-)}$ mapping dg categories $\C$ to pretriangulated categories $\ov{\C}$.

For future reference, we record the following basic lemma on ideals in dg categories, see \cite[Lemmas 2.1, 2.2, 2.3]{LM2}.

\begin{lemma}\label{easyideals}
Let $\C$ be a dg category which is a full subcategory of a dg category $\C'$. Let $\I$ be a dg ideal in $\C$.
\begin{enumerate}[(a)]
\item\label{idealtrivial}The restriction to $\C$ of the dg ideal generated by $\I$ in $\C'$ is equal to $\I$.
\item\label{idealonoverline} If $\C$ is a dg category equivalent to $\widehat{\{X\}}$ for some $X\in \C$, then $\I$ is generated by
the subset $\I\cap\End_{\C}(X)$.
\item \label{quotientpretri1}
If $\J$ is a dg ideal in $\ov{\C}$, then $\J=\ov{\left.\J\right|_{\C}}$, where $\left.\J\right|_{\C}$ is the restriction of $\J$ to $\C$.
\item \label{quotientpretri2}
There is a fully faithful dg functor $\ov{\C}/\ov{\I}\hookrightarrow \ov{\big(\C/\I\big)}$.
\item \label{quotientpretri3}
Assume that $\I$ has the property that if $\del(f)$ is in $\I$, then $f$ itself is in $\I$. Then the dg functor from \eqref{quotientpretri2} becomes a dg equivalence. 
\end{enumerate}
\end{lemma}

Moreover, we recall \cite[Lemma 2.4]{LM2} for later reference.

\begin{lemma}\label{Ccirctri}
If $\C$ is a pretriangulated category, then the dg idempotent completion $\C^\circ$ is also pretriangulated and $\iota\colon \C\to\C^\circ$ displays $\C$ as a full pretriangulated subcategory. In particular, this implies $\ov{\C}^\circ\simeq\ov{\ov{\C}^\circ}$ for any dg category $\C$.
\end{lemma}

Let $\C$ be a dg category. We denote the {\bf homotopy category} of $\C$ by $\K(\C)$, which has the same objects as $\C$ but whose morphism spaces are given by
$$\Hom_{\K(\C)}(X,Y)=H^0(\Hom_{\C}(X,Y)).$$
Recall that $\K(\C)$ is a triangulated category provided that $\C$ is pretriangulated  \cite[\S 1, Proposition~2]{BK}, \cite[Section~2.2]{Ke1}. Any dg functor $F\colon\C\to \D$ induces a functor $\K(F)\colon \K(\C)\to \K(\D)$, which is a triangle functor  \cite[\S 3]{BK} provided that $\C$ and $\D$ are pretriangulated.

\section{Dg \texorpdfstring{$2$}{2}-categories and \texorpdfstring{$2$}{2}-representations}\label{dg2}

In this section, we recall the main definitions  surrounding dg $2$-categories and define various notions associated pretriangulated $2$-representations.

\subsection{Dg \texorpdfstring{$2$}{2}-categories}\label{dg2cats}

A $2$-category $\cC$ is a {\bf dg $2$-category} if, for any pair of objects $\ti,\tj\in \cC$, the categories $\cC(\ti,\tj)$ are   dg categories and horizontal composition is a dg functor.

For a dg $2$-category $\cC$, on associates the dg $2$-category $\ov{\cC}$ of {\bf (one-sided) twisted complexes}. It 
\begin{itemize}
\item has the same objects as $\cC$;
\item its $1$-morphism categories are given by $\overline{\cC(\ti,\tj)}$;
\item horizontal composition of two $1$-morphisms $\rX=(\bigoplus_{m=1}^s\rF_m,\alpha)\in \overline{\cC(\tk,\tl)}$ and $X'=(\bigoplus_{n=1}^t \rF'_n,\alpha')\in\overline{\cC(\tj,\tk)}$ is given by
\begin{align}\label{orderingconvention}
\rX\circ \rX'=\left( \bigoplus_{(m,n)}\rF_m\rF_n', (\delta_{k',l'}\alpha_{k,l}\circ_0\id_{\rF'_{k'}}+\delta_{k,l}\id_{\rF_k}\circ_0\alpha'_{k',l'})_{(k,k'),(l,l')} \right),
\end{align}
where pairs $(m,n)$ are ordered lexicographically; 
\end{itemize}
This composition of $1$-morphisms gives a strict operation (for details, see \cite[Proposition 3.5]{LM}).
A dg $2$-category $\cC$ is called {\bf pretriangulated} if the embedding $\cC\hookrightarrow \ov{\cC}$ is part of a dg biequivalence.
Moreover, we say that a dg $2$-category $\cC$ {\bf  has generators} if each $\ov{\cC(\ti,\tj)}$ has a generator.

For a dg $2$-category $\cC$, the {\bf dg idempotent completion} $\cC^\circ$ is the dg $2$-category on the same objects as $\cC$, but with morphism categories $\cC^{\circ}(\ti,\tj)=\cC(\ti,\tj)^{\circ}$, the completion under dg idempotents, for any two objects $\ti,\tj$.  We note that Lemma \ref{Ccirctri} implies that if $\cC$ is pretriangulated, then $\cC^\circ$ is pretriangulated.

Given a dg $2$-category $\cC$, we define the $2$-category $\cZ\cC$, which has the same objects as $\cC$ and morphism categories $\Z(\cC(\ti,\tj))$.
Similarly, we define the {\bf homotopy $2$-category} $\cK\cC$ associated to $\cC$ as having the same objects as $\cC$ and morphism categories $\K(\cC(\ti,\tj))$.

\subsection{Pretriangulated \texorpdfstring{$2$}{2}-representations and their homotopy \texorpdfstring{$2$}{2}-representations}\label{dg2reps}

We now introduce pretriangulated $2$-representations, the class of $2$-representations studied in this article and \cite{LM2}.
As a target for pretriangulated $2$-representations, we define the dg $2$-category $\csfcat$ as the $2$-category whose
\begin{itemize}
\item objects are  pretriangulated categories $\C$;
\item $1$-morphisms  are dg functors;
\item $2$-morphisms are all morphisms between such dg functors.
\end{itemize}

By a {\bf pretriangulated $2$-representation} of a dg $2$-category $\cC$, we mean a $2$-functor $\bfM\colon \cC \to \csfcat$ such that the component functors from $\cC(\ti,\tj)$ to $\csfcat(\bfM(\ti),\bfM(\tj))$ are dg functors.
Explicitly, $\bfM$ maps 
\begin{itemize}
\item any object $\ti\in \cC$ to a pretriangulated category $\bfM(\ti)$,
\item any $1$-morphism $\rF \in \cC(\ti,\tj)$ to a dg functor $\bfM(\rF)\colon \bfM(\ti)\to\bfM(\tj)$,
\item any $2$-morphism $\alpha\colon  \rF\to\rG \in \cC(\ti,\tj)$ to a morphism 
$\bfM(\alpha)\colon \bfM(\rF)\to\bfM(\rG)$ of dg functors.
\end{itemize}
Moreover, 
$\bfM(\rG\rF)=\bfM(\rG)\bfM(\rF)$ and $\bfM(\one_\ti)=\one_{\bfM(\ti)}$, for $1$-morphisms $\rF, \rG$ and objects $\ti$.

If we have fixed a $2$-representation $\bfM$, we often write $\rF (X)$ or $\rF\, X$ instead of $\bfM(\rF)(X)$, and $\rF(f)$ instead of $\bfM(\rF)(f)$,  for any object $X$ and any morphism $f$ in $\bfM(\ti).$

A {\bf morphism of pretriangulated $2$-representations } $\Phi\colon\bfM\to\bfN$ consists of the data of 
\begin{itemize}
\item  a dg functor $\Phi_\ti\colon \bfM(\ti)\to \bfN(\ti)$ for each $\ti\in\cC$;
\item natural dg isomorphisms $\eta_\rF\colon \Phi_\tj\circ \bfM(\rF)\longrightarrow \bfN(\rF)\circ \Phi_\ti$ for each  $\rF\in\cC(\ti,\tj)$ such that, for composable $\rF,\rG$,
\begin{equation*}
\eta_{\mathrm{F}\mathrm{G}}=(\mathrm{id}_{\mathbf{N}(\mathrm{F})}\circ_0\eta_{\mathrm{G}})\circ_1
(\eta_{\mathrm{F}}\circ_0\mathrm{id}_{\mathbf{M}(\mathrm{G})}).
\end{equation*}
\end{itemize}

The collection of all pretriangulated $2$-representations of a dg $2$-category $\cC$, together with morphisms of pretriangulated $2$-representations and modifications satisfying the same condition as in \cite[Section 2.3]{MM3}, forms a dg $2$-category, which we denote by $\cC\tworep$.

\begin{definition}
The {\bf annihilator} $\Ann_{\cC}(\bfM)$ of a pretriangulated $2$-representation $\bfM$ is the dg $2$-ideal of $\cC$ consisting of those $2$-morphisms $\alpha$ with $\bfM(\alpha)=0$.
\end{definition}

We call a pretriangulated $2$-representation $\bfM$ {\bf acyclic} if any $X\in \bfM(\ti)$ is acyclic (that is, if there exists an endomorphism $f$ of $X$ with $\del(f)=\id_X$), for any object $\ti$.

We recall that any pretriangulated $2$-representation $\bfM$ of a dg $2$-category $\cC$ naturally extends to a pretriangulated $2$-representation $\bfM$ of $\ov{\cC}$. One easily verifies the following statement. 

\begin{lemma}\label{2rep-ov}
Extension from $\cC$ to $\ov{\cC}$ and restriction from $\ov{\cC}$ to $\cC$ define mutually inverse dg biequivalences between $\cC\tworep$ and $\ov{\cC}\tworep$.
\end{lemma}

For any dg $2$-category $\cC$  and object $\ti\in \cC$, the {\bf $\ti$-th principal pretriangulated $2$-representation} $\bfP_\ti$ is the pretriangulated $2$-representation which maps 
\begin{itemize}
\item any object $\tj$ to $\overline{\cC(\ti, \tj)}$,
\item any $1$-morphism $\rF\in \cC(\tj, \tk)$ to the dg functor from $\overline{\cC(\ti, \tj)}$ to $\overline{\cC(\ti, \tk)}$ induced by composition with $\rG$, 
\item any $2$-morphism to the induced dg morphism of functors.
\end{itemize}

We can further extend pretriangulated $2$-representations to the dg idempotent completion $\cC^\circ$ of $\cC$. Indeed, for $\bfM\in \cC\tworep$, we can define $\bfM^\circ \in \cC\tworep$ as given by $\bfM^\circ(\ti)=\bfM(\ti)^\circ$ and extend the dg functors $\bfM(\rF)$ to $\bfM^{\circ}(\rF)$ using \eqref{functorcirc}. The resulting pretriangulated $2$-representation $\bfM^\circ$ extends to a pretriangulated $2$-representation of $\cC^\circ$, which sends $\rF_e$, for $e\colon \rF\to \rF$ a dg idempotent in $\cC(\ti,\tj)$, to the dg functor $\bfM^\circ(\rF_e)\colon \bfM^\circ(\ti)\to \bfM^\circ(\tj)$ given by
$$X\mapsto \bfM^\circ(\rF_e)(X)=(\bfM(\rF)(X))_{\bfM(e)_X}, \qquad f\mapsto \bfM(e)_Y\circ \bfM(\rF)(f)\circ \bfM(e)_X.$$
Here $X,Y$ are objects and $f\colon X\to Y$ is a morphism in $\bfM^\circ(\ti)$. Functoriality of $\bfM(\rF_e)$ follows from naturality of $\bfM(e)$. We recall the following lemma.

\begin{lemma}\cite[Lemma 3.6]{LM2}\label{extendtocirc}
There is a dg $2$-functor 
$$(-)^\circ\colon \cC\tworep \to \cC^\circ \tworep,$$
where the assignments on pretriangulated $2$-representations, morphisms, and modifications extend the given structures from $\cC$ to $\cC^\circ$.
\end{lemma}

The dg $2$-functor $(-)^\circ\colon \cC\tworep \to \cC^\circ \tworep$ can be interpreted as a left dg $2$-adjoint to restriction $\bfN\mapsto \left.\bfN\right|_{\cC}$ along the inclusion of $\cC$ into $\cC^\circ$ in the sense that there are 
 dg equivalences 
$$\adj{E}{\Hom_{\cC}(\bfM,\left.\bfN\right|_{\cC})}{\Hom_{\cC^\circ}(\bfM^\circ,\bfN)}{R},$$
for a dg idempotent complete pretriangulated $2$-representation  $\bfN$ of $\cC^\circ$ and $\bfM\in \cC\tworep$. 
In particular, $(-)^\circ$ restricts to a dg biequivalence  on the full $2$-subcategories of dg idempotent complete   pretriangulated $2$-representations.

We can associate {\bf homotopy $2$-representations} to pretriangulated $2$-representations.
Indeed, given a pretriangulated $2$-representation $\bfM$ of a dg $2$-category $\cC$, set
 $\bfK\bfM(\ti):=\K(\bfM(\ti))$ for any $\ti\in\cC$ and 
$$\bfK\bfM(\rF):= \K(\bfM(\rF))\colon \bfK\bfM(\ti)\to \bfK\bfM(\tj),$$
for any $1$-morphism in $\cC(\ti,\tj)$. A dg $2$-morphism $\alpha\colon \rF\to \rG$ induces a natural transformation $\bfK\bfM(\alpha)\colon \bfK\bfM(\rF)\to \bfK\bfM(\rG)$. By construction, $\bfK\bfM$ is an additive $2$-representation of $\cZ\cC$.
In fact,  $\bfK\bfM$ is an additive $2$-representation of the homotopy $2$-category $\cK\cC$.

\subsection{Dg ideals of  \texorpdfstring{$2$}{2}-representations and pretriangulated \texorpdfstring{$2$}{2}-subrepresentations}\label{idealsect}
In this section, we discuss various notions associated to dg ideals, required later on.

A {\bf dg ideal} of a pretriangulated $2$-representation $\bfM$ of a dg $2$-category $\cC$ consists a collection of ideals $\bfI(\ti)\subset \bfM(\ti)$, which is closed under differentials and the $\cC$-action.

\begin{definition}\label{idealgen}
Given a pretriangulated $2$-representation $\bfM$ of $\cC$, an object $\ti\in \cC$, and a morphism $f$ in $\bfM(\ti)$, we define the {\bf dg ideal $\bfI_{\bfM}(f)$ generated} by $f$ to be the smallest dg ideal of $\bfM$ such that $f$ is contained in $\bfI_{\bfM}(f)(\ti)$. 
\end{definition}

For any dg ideal $\bfI$ in a pretriangulated $2$-representation $\bfM$, we define the quotient $\bfM/\bfI$ as acting on 
$$\left(\bfM/\bfI\right)(\ti)=\ov{\bfM(\ti)/\bfI(\ti)}.$$
In particular, for a dg $2$-subrepresentation $\bfN$ of  $\bfM$, we define the {\bf quotient} $\bfM/\bfN$ as the quotient of $\bfM$ by the dg ideal generated by $\bfN$.

Given  a pretriangulated $2$-representation $\bfM$ of $\cC$, we say that a pretriangulated {\bf $2$-sub\-rep\-re\-sentation} $\bfN$ of $\bfM$ is a collection of thick subcategories $\bfN(\ti)\subseteq \bfM(\ti)$, for $\ti$ of $\cC$, such that, for any $1$-morphisms $\rG$ and $2$-morphism $\alpha$ in $\cC(\ti,\tj)$, we have 
\begin{align*}\bfN\rG(N)=\bfM\rG(N) &\in \bfN(\tj), & \bfN \rG(f)&=\bfM\rG(f)\in \bfN(\tj),&\bfN(\alpha)_N &=\bfM(\alpha)_N, \end{align*}
for any object $N$ and morphism $f$ in $\bfN(\ti)$. In particular, the $\bfN(\ti)$ are closed under the $\cC$-action. We emphasize that we do require each $\bfN(\ti)$ to be closed under forming biproducts and taking dg direct summands that exist in $\bfM(\ti)$.

For a collection of pretriangulated $2$-subrepresentations $\lbrace \bfN^\nu\rbrace_{\nu\in I}$ of a given pretriangulated $2$-representation $\bfM$, we can define the {\bf sum} $\sum_{\nu\in I} \bfN^\nu$ as the smallest pretriangulated $2$-subrepresentation of $\bfM$ which contains all $\bfN^\nu$.

\begin{definition}\label{2repgen}
For any pretriangulated $2$-representation $\bfM$ of $\cC$ and an object $X\in \bfM(\ti)$, for some $\ti\in \cC$, we denote the $2$-subrepresentation on the thick closure of $\{\bfM(\rG) X | \rG \in \cC(\ti,\tj), \tj\in \cC\}$ inside $\coprod_{\tj\in\cC}\bfM(\tj)$ by $\bfG_\bfM(X)$ . We say that $\bfG_\bfM(X)$ is the dg $2$-subrepresentation {\bf $\cC$-generated} by $X$. If $\bfG_\bfM(X)=\bfM$, the object $X$ is said to $\cC$-generate $\bfM$. A pretriangulated $2$-representation $\bfM$ is called {\bf cyclic} if it is $\cC$-generated by some $X\in \bfM(\ti)$, for some $\ti\in \cC$.

Similarly, given an object $X$ in $\bfM(\ti)$, define $\bfG_{\bfK\bfM}(X)$ to be the $2$-subrepresentation of $\bfK\bfM$ with underlying category associated to $\tj$ given by the thick closure of $\{\bfK\bfM(\rG)(X)\,|\, \rG\in \cC(\ti,\tj)\}$, where by thick closure we mean the full triangulated subcategory closed under direct sums and summands.

Finally, we say that $\bfM$ is {\bf weakly $\cC$-generated} by an object $X\in \bfM(\ti)$ if the ideal generated by $\id_X$ equals $\bfM$.
\end{definition}

For future use, we record some  basic results on ideals of dg $2$-subrepresentations.
Let $\bfM$ be a pretriangulated $2$-representation of $\cC$ and $\bfN$ a pretriangulated  $2$-subrepresentation of $\bfM$. For any dg ideal $\bfI$ in $\bfM$, we denote by $\bfI\cap \bfN$ the restriction of $\bfI$ to $\bfN$, i.e., for any object $\ti$ of $\cC$ and objects $X,Y$ of $\bfN(\ti)$,
$$\Hom_{(\bfI\cap \bfN)(\ti)}(X,Y)=\Hom_{\bfI(\ti)}(X,Y).$$
Then $\bfI\cap \bfN$ is a dg ideal in $\bfN$.

\begin{lemma}\label{quotientfaithful}
The induced morphism of pretriangulated $2$-representations
$\bfN/(\bfI\cap \bfN)\longrightarrow \bfM/\bfI$
is fully faithful.
\end{lemma}
\begin{proof}
Let $\ti$ be an object of $\cC$. 
Then the kernel of the dg functor $\bfN(\ti)\to \bfM(\ti)/\bfI(\ti)$ is given 
$\bfI(\ti)\cap \bfN(\ti)$. Thus, the induced dg functor from  
$\bfN(\ti)/(\bfI\cap \bfN)(\ti)$ to $\big(\bfM/\bfI\big)(\ti)$ is faithful.

By definition of pretriangulated $2$-subrepresentations, $\bfN(\ti)$ is a full dg subcategory of $\bfM(\ti)$. Hence, it follows that $\bfN(\ti)/(\bfI\cap \bfN)(\ti)\to \bfM(\ti)/\bfI(\ti)$ is also full on morphism spaces. 

Thus the morphism  $\bfN/(\bfI\cap \bfN)\longrightarrow \bfM/\bfI$ is also fully faithful.
\end{proof}

\begin{lemma}\label{idealofsub}
Let $\bfM\in \cC\tworep$, $\bfN$ a pretriangulated $2$-subrepresentation of $\bfM$, and $\bfI$ a dg ideal in $\bfN$. Let $\bfI'$ be the ideal in $\bfM$ generated by $\bfI$. Then $\bfI'\cap \bfN=\bfI$.
\end{lemma}

\proof
Note that $\bfI'$ is automatically a dg ideal. 
The statement follows from the definition and the fact that $\bfN(\tj)$ is a pretriangulated subcategory of $\bfM(\tj)$.
\endproof

\begin{lemma}\label{2repidealonoverline}
Let $\cC$ be a dg $2$-category, $\bfM$ a pretriangulated $2$-representation of $\cC$, and $\bfI$ an ideal of $\bfM$.
Further suppose $X\in \bfM(\ti)$ (weakly) $\cC$-generates
 $\bfM$. Then $\bfI$ is completely determined by the sets $\Hom_{\bfI(\tj)}(\rF X,\rG X)$, where $\rF,\rG$ range over the $1$-morphisms in $\cC(\ti,\tj)$ for any $\tj\in \cC$.

In particular, if $\rG_{\ti,\tj}\in \cC(\ti,\tj)$ form a set of generators for each $\ti,\tj$, then the sets $\End_{\bfI(\tj)}(\rG_{\ti,\tj} X)$ determine $\bfI$.
\end{lemma}

\proof
Note that if $X$ $\cC$-generates $\bfM$, then it, in particular, weakly $\cC$-generates $\bfM$. 
Thus, we assume that $X\in \bfM(\ti)$ weakly $\cC$-generates $\bfM$. Then, for any $\ti$ and any object $Y\in \bfM(\tj)$, we find that $\id_Y$ is in the ideal $\bfI_{\bfM}(X)$. 

Thus, using Lemma \ref{easyideals}\eqref{idealonoverline}, $\bfI(\tj)$ is determined by  the sets $\bfI(\tj)\cap \Hom_{\bfM(\tj)}(\rF X,\rG X)$. 
The same proof works if $X\in \bfM(\ti)$ weakly $\cC$-generates $\bfM$, see Definition \ref{2repgen}, noting that if $g\in \Hom_{\bfI(\tj)}(M,N)$ for $M,N$ such that $\id_M = p_M\id_{\rF X}\iota_M, \id_N = p_N\id_{\rF Y}\iota_N$,  then $g'=\iota_N g p_M\in \Hom_{\bfI(\tj)}(\rF X,\rG X)$ and $g= p_Ng' \iota_M$.

 If each $\cC(\ti,\tj)$ has a generator $\rG_{\ti,\tj}$, it suffices to set $\rF=\rG=\rG_{\ti,\tj}$ observing that there are no $2$-morphisms between $\rG_{\ti,\tj}X$ and $\rG_{\ti,\tk}X$ for $\tj\neq \tk$.
\endproof

\subsection{Quotient-simple pretriangulated \texorpdfstring{$2$}{2}-representations}\label{sect-qs}

A pretriangulated $2$-representation $\bfM$ is called {\bf quotient-simple} if it does not have any proper nonzero dg ideals.

We will now show that passing to dg idempotent completions preserves quotient-simplicity.
Let $\bfM\in \cC\tworep$ and let $\bfI$ be a dg ideal in $\bfM$. Denote by $\bfI^\circ$ the dg ideal generated by $\bfI$ in the pretriangulated $2$-representation $\bfM^\circ$ defined in Lemma \ref{extendtocirc}. We use this notation for the ideal in $\bfM^\circ$ viewed as either a pretriangulated $2$-representation of $\cC$ or $\cC^\circ$. We record the following straightforward lemmas.

\begin{lemma}\label{circideal}
For objects $X,Y\in\bfM(\ti)$ and dg idempotents $e\colon X\to X$, $e'\colon Y\to Y$,  the morphism spaces in $\bfI^\circ$ are given by
$$\Hom_{\bfI^\circ(\ti)}(X_e,X_{e'})=e'\circ \Hom_{\bfI(\ti)}(X,Y)\circ e.$$
In particular, $\bfM\cap \bfI^{\circ}=\bfI$. 
\end{lemma}

We note that $\bfJ$ is a dg ideal in $\bfM^\circ\in \cC^\circ \tworep$  if and only if $\bfJ$ is a dg ideal in $\bfM^\circ\in \cC\tworep$. This follows from 
$$\bfM(\rF_e)(f)=(\bfM(e)(\id)\bfM(\rF)(f)(\bfM(e)(\id),$$
for any dg idempotent $e\colon \rF\to \rF$.
Hence, if $\bfJ(\ti)$ is a two-sided dg ideal in $\bfM(\ti)$ and preserved by $\bfM(\rF)$, then it is preserved by $\bfM(\rF_e)$.

\begin{lemma}\label{LemIcapM}
If $\bfJ$ is a dg ideal in $\bfM^\circ$, then $\bfJ$ is generated by the subset $\bfJ\cap \bfM$. 
\end{lemma}

As a consequence, completing under dg idempotents preserves quotient-simplicity. 

\begin{lemma}\label{lem-simpletrans-circ}
Let $\bfM$ be a quotient-simple pretriangulated $2$-representation of $\cC$. Then $\bfM^\circ$ is a quotient-simple  pretriangulated $2$-representation of $\cC^\circ$ (as well as $\cC$).
\end{lemma}
\begin{proof}
Assume that $\bfM$ is quotient-simple and let $\bfJ$ be a proper dg ideal in $\bfM^\circ$. Then $\bfI=\bfJ\cap \bfM$ is a proper dg ideal in $\bfM$ as, otherwise, $\bfI\subseteq \bfI^\circ$ would contain all identities in $\bfM$, and hence all identities of $\bfM^\circ$ as well. Thus, $\bfI$ is the zero ideal. This implies that $\bfJ$ is the zero ideal as $\bfI$ generates $\bfJ$ as an ideal. Hence, $\bfM^\circ$ is quotient-simple. 
\end{proof}

\section{Dg cell \texorpdfstring{$2$}{2}-representations}\label{cellchapter}

In finitary $2$-representation theory, a crucial role has been played by the construction of so-called cell $2$-representations. In many, but not all cases, these exhaust simple $2$-representations, and in a certain sense, all simple $2$-representations of a fiat $2$-category are governed by them. Cell $2$-representations are constructed by taking the $2$-subrepresentation, generated by the $1$-morphisms in a so-called {\em left cell}, which is combinatorially defined, of a principal (i.e.\ representable when viewed as a pseudofunctor) $2$-representation, and quotienting out by what turns out to be its {\em unique} maximal ideal. For the left cell containing the identity $1$-morphism on an object $\ti$, this $2$-subrepresentation is precisely the principal $2$-representation $\bfP_\ti$ itself. 

In our more general setting, the following example shows that uniqueness of maximal ideals cannot be expected.

\begin{example}\label{geomex1}
Let $R$ be a finitely generated commutative ring and $\C$ the dg category with single object $\one$ whose endomorphism ring is $R$. Let $\cC$ be the dg $2$-category on one object $\bullet$, where $\cC(\bullet,\bullet)$ is $\ov{\C}$. Note that $\cC(\bullet,\bullet)$ is dg equivalent to the dg category of bounded complexes of free $R$-modules and under this equivalence, horizontal composition translates to the tensor product over $R$.

Consider the principal $2$-representation $\bfP_{\bullet}$. It is easy to see that every $\mathbf{m}\in \MaxSpec(R)$ defines a maximal ideal in $\bfP_{\bullet}$ and that the category underlying the corresponding quotient $2$-representation is given by bounded complexes over $R/\mathbf{m}$.
Thus, we obtain at least as many quotient-simple $2$-representations associated to $\bfP_{\bullet}$ as there are maximal ideals in our ring. This in contrast to the finitary case, where uniqueness maximal ideals in the $2$-representation is implied by locality of the endomorphism rings of the generating $1$-morphisms. 
\end{example}

Motivated by this example, we will generalise the approach to cell $2$-representations in this section. In Section \ref{cellcomb}, we will define two natural candidates ({\em weak} versus {\em strong}) for cells. The first is justified in Section \ref{apexsubsec}, where we prove that to any quotient-simple $2$-representation we can uniquely associate a certain weak two-sided cell, its so-called weak apex, generalising the analogous statement in finitary $2$-representation theory and reducing the problem of classifying quotient-simple $2$-representations to a cell-by-cell approach. This justifies the definition of weak cells, while the strong cells connect more naturally to the associated triangulated $2$-representations, see Section \ref{sectriangcell}. After defining the $2$-subrepresentations corresponding to weak and strong cells in Section \ref{subsforcells}, we then define spectra of maximal ideals of these $2$-subrepresentations in Section \ref{secspectra} and prove a bijection between maximal spectra of associated weak and strong cells in Theorem \ref{maxspecbij}. After defining cell $2$-representations in Definition \ref{def:cell2rep} and returning to the above example in Example \ref{geomex},  the rest of Section \ref{cellchapter} is devoted to a technical analysis of weak and strong cell $2$-representations.

Throughout this section, let $\cC$ be a pretriangulated $2$-category, which is not a genuine restriction by Lemma \ref{2rep-ov}.

\subsection{Cell combinatorics}\label{cellcomb}

We write $\mathcal{S}(\cC)$ for the set of dg isomorphism classes of dg indecomposable $1$-morphisms in $\cC$ up to shift. This set forms a multi-semigroup, see \cite[Section~3]{MM2}, and can be equipped with several natural preorders inspired by \cite{Gr}. 

We now give two definitions of left orders on $\mathcal{S}(\cC)$, using the notation from Definitions \ref{idealgen} and \ref{2repgen}.

\begin{definition}\label{def:cells} Let $\rF\in \cC(\ti,\tj)$ and $\rG\in \cC(\ti,\tk)$.
\begin{enumerate}[(a)]
\item\label{def1} We say $\rF$ is {\bf left strongly less than or equal to} $\rG$ if $\bfG_{\bfP_\ti}(\rG)\subseteq \bfG_{\bfP_\ti}(\rF)$. If this is the case, we write $\rF\leq_{L}\rG$.
\item\label{def2} We say $\rF$ is {\bf left weakly less than or equal to} $\rG$ if $\bfI_{\bfP_\ti}(\id_\rG)\subseteq \bfI_{\bfP_\ti}(\id_\rF)$.  If this is the case, we write $\rF\preceq_{L}\rG$.
\end{enumerate}
\end{definition}

\begin{remark}\label{explicit}
Observe that $\rF\preceq_{L}\rG$ if and only if $\id_\rG\in \bfI_{\bfP_\ti}(\id_\rF)$ 
or, in other words, if $\id_\rG=p\circ\id_{\rH\rF}\circ \iota$ for some $1$-morphism $\rH$ and $2$-morphisms $p$ and $\iota$ in $\cC$, which are not necessarily annihilated by the differential. Note that the definition of weak order does not depend on the differential graded structure, but rather on the underlying $\Bbbk$-linear structure.
Moreover, $\rF\leq_{L}\rG$ if and only if $\rG$ is in the thick closure of $\{\rH\rF \,\vert \,\rH \text{ a $1$-morphism in } \cC\}$.
\end{remark}

Equivalence classes for these two partial preorders $\leq_L$ and $\preceq_L$ are called {\bf strong left cells} and {\bf weak left cells}, respectively. If $\rF$ and $\rG$ are strongly left equivalent, we write $\rF\steq_L\rG$, and if they are weakly left equivalent, we write $\rF\sim_L\rG$. If $\rF\leq_{L}\rG$ and $\rF\ngeq_{L}\rG$, we will write $\rF<_{L}\rG$.  Similarly, for $\rF\preceq_{L}\rG$ and $\rF\nsucceq_{L}\rG$, we will write $\rF\prec_{L}\rG$.

Using Remark \ref{explicit}, we can now define the {\bf strong} and {\bf weak right} preorders $\leq_R$ and  $\preceq_R$, as well as 
{\bf strong} and {\bf weak two-sided} preorders $\leq_J$ and $\preceq_J$.

\begin{definition}
\begin{enumerate}[(a)]
\item\label{rdef1} We say $\rF$ is {\bf right strongly less than or equal to} $\rG$ if $\rG$ is in the thick closure of $\{\rF\rH \,\vert \,\rH \text{ a $1$-morphism in } \cC\}$. If this is the case, we write $\rF\leq_{R}\rG$.
\item\label{rdef2} We say $\rF$ is {\bf right weakly less than or equal to} $\rG$ if $\id_\rG=p\circ\id_{\rF\rH}\circ \iota$ for some $1$-morphism $\rH$ and $2$-morphisms $p$ and $\iota$ in $\cC$.   If this is the case, we write $\rF\preceq_{R}\rG$.
\item\label{2def1} We say $\rF$ is {\bf two-sided strongly less than or equal to} $\rG$ if $\rG$ is in the thick closure of $\{\rH_1\rF\rH_2 \,\vert\, \rH_1,\rH_2 \text{ $1$-morphisms in } \cC\}$. If this is the case, we write $\rF\leq_{J}\rG$.
\item\label{2def2} We say $\rF$ is {\bf right weakly less than or equal to} $\rG$ if $\id_\rG=p\circ\id_{\rH_1\rF\rH_2}\circ \iota$ for some $1$-morphisms $\rH_1,\rH_2$ and $2$-morphisms $p$ and $\iota$ in $\cC$.   If this is the case, we write $\rF\preceq_{J}\rG$.\end{enumerate}
\end{definition}

\begin{remark}\label{weakin2cat}
As explained in Remark \ref{explicit}, the condition $\rF\preceq_L\rG$ is equivalent to the left dg $2$-ideal in $\cC$ generated by $\id_\rG$ being contained in the left dg $2$-ideal generated by $\id_\rF$. Similarly, $\rF\preceq_R\rG$ if and only if the right dg $2$-ideal in $\cC$ generated by $\id_\rG$ is contained in the right dg $2$-ideal generated by $\id_\rF$; and 
$\rF\preceq_J\rG$ if and only if the (two-sided) dg $2$-ideal in $\cC$ generated by $\id_\rG$ is contained in the (two-sided)  dg $2$-ideal in $\cC$ generated by $\id_\rF$.
\end{remark}

For these notions, we obtain the corresponding {\bf strong} and {\bf weak right} and {\bf strong} and {\bf weak two-sided} cells, respectively. 

Observe that $\leq_L$ (resp. $\preceq_L$) defines a genuine partial order on the set of strong (resp. weak) left cells, and similarly  for $\leq_R$ (resp. $\preceq_R$) and strong (resp. weak) right cells, and for $\leq_J$ (resp. $\preceq_J$) and strong (resp. weak) two-sided cells. For a fixed (strong or weak) left cell $\L$ in $\cC$ , notice that there exists a unique object $\ti_\L\in \cC$ such that each $\rF\in \L$ is in $\cC(\ti_\L,\tj)$ for some $\tj\in \cC$.

Note that $\rF\leq_{L}\rG$ implies $\rF\preceq_{L}\rG$.
Thus, $\rF\steq_L\rG$ implies $\rF\sim_L\rG$ and any weak left cell is a union of strong left cells.

\subsection{The weak apex}\label{apexsubsec}

In this subsection, we follow the ideas from \cite{ChMa}. For any $\bfM\in \cC\tworep$, set $\cC_{\bfM}=\cC/\Ann_{\cC}(\bfM)$.

\begin{lemma}
Strong, respectively weak, two-sided cells of $\cC_{\bfM}$ form a subset of strong, respectively weak, two-sided cells of $\cC$.
\end{lemma}

\proof
Let $\J$ be a strong two-sided cell. If $\rF\in \J$ is annihilated by $\bfM$, then $\rG$ is annihilated by $\bfM$ for all $\rG\geq_J \rF$. Similarly, let  $\J$ be a weak two-sided cell and $\rF\in \J$. If $\id_\rF$ is in $\Ann_{\cC}(\bfM)$, the so is $\id_\rG$ for every $\rG\succeq_J \rF$, since the annihilator is a  dg $2$-ideal of $\cC$. The statement of the lemma follows.
\endproof

\begin{proposition}\label{uniquemax}
If $\bfM\in \cC\tworep$ is quotient-simple, then $\cC_{\bfM}$ has a unique maximal weak two-sided cell.
\end{proposition}
\proof
Suppose both $\J$ and $\J'$ are two distinct maximal weak two-sided cells of $\cC_{\bfM}$. 
Let $X\in\bfM(\ti)$ be nonzero. Then there exists $\rF\in \J$ such that $\rF X\neq 0$. 
Indeed, assume, for a contradiction, that $\rF X= 0$ for all $\rF\in \J$. We claim that $\bfI_{\bfM}(\id_X)$ is a proper ideal in $\bfM$. Let $Y\in \bfM(\tj)$ be such that $\id_Y\in \bfI_{\bfM}(\id_X)$. Thus, w.l.o.g., $\id_Y=p\circ \id_{\rH X}\circ \iota$, for some $1$-morphism $\rH$ of $\cC$. Hence, 
$$\bfM(\rF)(\id_Y)=\bfM(\rF)(p)\circ\id_{\rF\rH X}\circ \bfM(\rF)(\iota)=0,$$
since $\rF\rH X=0$ by maximality of $\J$. This shows that $\bfI_{\bfM}(\id_X)$ has to be proper as $\J$ is not contained in $\Ann_{\cC}(\bfM)$. This contradicts quotient-simplicity of $\bfM$. 

By the same argument, there exists $\rG\in \J'$ with $\rG\rF X\neq 0$. Since $\id_{\rG\rF}\in \bfI_{\bfP_\ti}(\id_\rF)$, any dg indecomposable direct summand $\rH$ of $\rG\rF$ satisfies $\rH\succeq_L \rF$ and hence $\rH\succeq_J \rF$. Similarly,  
any dg indecomposable direct summand $\rH$ of $\rG\rF$ satisfies $\rH\succeq_R \rG$ and hence $\rH\succeq_J \rG$. Again, due to maximality of $\J$ and $\J'$, we deduce $\rH\in \J\cap\J'$, a contradiction.
\endproof

\begin{definition}
If $\bfM\in \cC\tworep$ is quotient-simple, we call the unique maximal weak two-sided cell in $\cC_{\bfM}$ the {\bf weak apex} of $\bfM$.
\end{definition}

\subsection{Pretriangulated \texorpdfstring{$2$}{2}-subrepresentations of principal dg \texorpdfstring{$2$}{2}-representations}\label{subsforcells}

Associated to the two left orders from Definition \ref{def:cells}, we obtain two possibly different dg $2$-subrepresentations of $\bfP_\ti$. 

\begin{definition} For a $1$-morphism $\rF$ in $\cC$, we define $\bfR^{\leq}_{\rF} = \bfG_{\bfP_\ti}(\rF)$. We further define 
$\bfR^{\preceq}_{\rF}$ to be the unique maximal pretriangulated $2$-subrepresentation of $\bfP_\ti$ such that $\bfR^{\preceq}_{\rF}$ is contained in $\bfI_{\bfP_\ti}(\id_\rF)$. 
\end{definition}

\begin{lemma}\label{indep} Let $\rF\in \cC(\ti,\tj)$ and $\rG\in \cC(\ti,\tk)$ be dg indecomposable.
\begin{enumerate}[(a)]
\item\label{indep1} If  $\rF\leq_L\rG$, then $\bfR^{\leq}_{\rF}\supseteq\bfR^{\leq}_{\rG}$. In particular, if  $\rF\steq_L\rG$, then $\bfR^{\leq}_{\rF}=\bfR^{\leq}_{\rG}$.
\item\label{indep2} If  $\rF\preceq_L\rG$, then $\bfR^{\preceq}_{\rF}\supseteq\bfR^{\preceq}_{\rG}$. In particular, if 
$\rF\sim_L\rG$, then $\bfR^{\preceq}_{\rF}=\bfR^{\preceq}_{\rG}$.
\end{enumerate}
\end{lemma}

\proof
The statement in \eqref{indep1} follows immediately from the definition. Similarly, the second statement in  \eqref{indep2} follows from the first, so we only need to show that $\rF\preceq_L\rG$ implies $\bfR^{\preceq}_{\rF}\supseteq\bfR^{\preceq}_{\rG}$.
By definition, $\rF\preceq_L\rG$ if and only if $\bfI_{\bfP_\ti}(\id_\rG)\subseteq \bfI_{\bfP_\ti}(\id_\rF)$. The categories $\bfR^{\preceq}_{\rF}(\tj)$ are, however, precisely the full subcategories of  $\bfP_\ti(\tj)$ consisting of those objects whose identities are in $\bfI_{\bfP_\ti}(\id_\rF)$. The claim follows.
\endproof

Justified by the above lemma, we will often write $\bfR^{\leq}_{\L}$ if $\L$ is a strong left cell, and  $\bfR^{\preceq}_{\L}$ if $\L$ is a weak left cell.

The next lemma explains the relationship between the pretriangulated $2$-representations generated by a weak cell and a strong cell contained in it.

\begin{lemma}\label{Rdifference}
If $\L$ is a strong left cell which is contained in a weak left cell $\L'$, then $\bfR^{\leq}_{\L}$ is a pretriangulated $2$-subrepresentation of  $\bfR^{\preceq}_{\L'}$. Moreover, $Y\in \bfR^{\preceq}_{\L'}(\tj)$ if and only if $\id_Y=p\circ \id_{\rF X} \circ \iota$ for $X\in \L$ and $\rF$ a $1$-morphism in $\cC$.
\end{lemma}

\proof
The first statement is clear. For the second statement, let $\Y$ be the full subcategory of $\bfP_{\ti}(\tj)$ consisting of objects $Y$ such that $\id_Y=p\circ \id_{\rF X} \circ \iota$ for $X\in \L$ and $\rF$ a $1$-morphism in $\cC$. It follows directly from the definition, that $\Y\subseteq \bfR^{\preceq}_{\L'}(\tj)$.

Conversely, let $Y\in \bfR^{\preceq}_{\L'}(\tj)$. This implies that $\id_Y$ is in the ideal $\bfI_{\bfP_\ti}(\id_X)$ for $X\in \L\subseteq \L'$. A morphism in $\bfI_{\bfP_\ti}(\id_X)$ is of the form 
$\sum_l a_l\circ \rF_l(\id_X)\circ b_l$ for morphisms $a_l, b_l \in \bfP_{\ti}(\tj)$ and $\rF_l\in \cC(\ti,\tj)$. 
Hence, we can write $\id_Y=a\circ \diag(\rF_l(\id_X) )\circ b$ where $a$ and $b$ are the row respectively column vectors of the $a_l$ respectively $b_l$. This completes the proof as $\diag(\rF_l(\id_X) )=\id_{\oplus \rF_l X}$.
\endproof

In the following, we compare ideals in $\bfR^{\leq}_{\L}$ and $\bfR^{\preceq}_{\L'}$.

\begin{lemma}\label{downup}
Let $\L$ be a strong left cell contained in a weak left cell $\L'$ and $\bfI$ be a dg ideal in $\bfR^{\preceq}_{\L'}$. Then the ideal generated by $\bfI\cap\bfR^{\leq}_{\L}$ inside $\bfR^{\preceq}_{\L'}$ equals $\bfI$.
\end{lemma}

\proof
It is clear that $\bfI$ contains the  ideal generated by $\bfI\cap\bfR^{\leq}_{\L}$ inside $\bfR^{\preceq}_{\L'}$. Conversely, if $\alpha\colon Y\to Y' \in \bfI(\tj)$ for some $\tj\in \cC$, then $\id_Y=p\circ\id_{\rF X}\circ \iota$, and $\id_{Y'}=p'\circ\id_{\rF' X}\circ \iota'$ for some $\rF X$, $\rF' X$ in $\bfR_{\L}^\leq(\tj)$ as in Lemma \ref{Rdifference}. Hence $\iota'\circ\alpha\circ p \in \bfI\cap \bfR_{\L}^\leq(\tj)$, and therefore, $\alpha = p'\circ\iota'\circ\alpha\circ p\circ \iota$ is in the ideal generated by $\bfI\cap \bfR_{\L}^\leq$. The claim follows.
\endproof

\subsection{Spectra of cells}\label{secspectra}

For $\bfM$ a pretriangulated $2$-representation, we define the {\bf spectrum} $\Spec(\bfM)$ to be the set of proper dg ideals of $\bfM$ and $\MaxSpec(\bfM)$ to be the subset consisting of those dg ideals which are maximal with respect to inclusion. In the following, we will compare $\MaxSpec(\bfR_\L^\leq)$ and $\MaxSpec(\bfR_{\L'}^\preceq)$, where $\L$ is a strong left cell contained in a weak left cell $\L'$.

\begin{lemma}\label{proper}
Assume $\L$ is a strong left cell which is contained in a weak left cell $\L'$, and $\bfI\in \Spec(\bfR_{\L}^\leq)$. Let $\bfI'$ be the ideal in $\bfR_{\L'}^\preceq$ generated by $\bfI$. Then $\bfI'$ is a proper dg ideal of $\bfR_{\L'}^\preceq$.
\end{lemma}

\proof
If $\bfI'$ contained all identities on objects in $\bfR_{\L'}^\preceq(\tj)$ for all objects $\tj$, then it would also contain all identities on objects in $\bfR_{\L}^\leq(\tj)$ for all objects $\tj$. This contradicts the fact that $\bfI=\bfI'\cap \bfR_{\L}^\leq$  (see Lemma \ref{idealofsub}) is in $\Spec(\bfR_{\L}^\leq)$.
\endproof

\begin{lemma}\label{goingup}
Assume $\L$ is a strong left cell which is contained in a weak left cell $\L'$, and $\bfI\in \MaxSpec(\bfR_{\L}^\leq)$. Let $\bfI'$ be the ideal in $\bfR_{\L'}^\preceq$ generated by $\bfI$. Then $\bfI'\in \MaxSpec(\bfR_{\L'}^\preceq)$.
\end{lemma}

\proof By Lemma \ref{proper}, $\bfI' \in \Spec(\bfR_{\L'}^\preceq)$. Suppose $\bfJ \in \Spec(\bfR_{\L'}^\preceq)$ strictly containing $\bfI'$. Consider $\bfJ\cap  \bfR_{\L}^\leq$. This ideal contains $\bfI$, hence by maximality of $\bfI$, $\bfJ\cap  \bfR_{\L}^\leq$ is either equal to $\bfI$ or $\bfR_{\L}^\leq$.

If $\bfJ\cap  \bfR_{\L}^\leq=\bfR_{\L}^\leq$, it contains $\id_X$ for any $X\in \L$. However, since $\L\subseteq \L'$, by definition,  $\bfR_{X}^\preceq=\bfR_{\L'}^\preceq$, so $\bfJ=\bfR_{\L'}^\preceq$, which contradicts $\bfJ\in \MaxSpec(\bfR_{\L'}^\preceq)$. Hence $\bfJ\cap  \bfR_{\L}^\leq=\bfI$. 

Suppose $\alpha\colon Y\to Y'$ is a morphism in $\bfJ$, which is not in $\bfI'$. Then $\id_Y=p\circ\id_{\rF X}\circ \iota$, and $\id_{Y'}=p'\circ\id_{\rF' X}\circ \iota'$ as in Lemma \ref{Rdifference}, and hence $\alpha= p'\circ\id_{\rF' X}\circ \iota'\circ \alpha \circ p\circ\id_{\rF X}\circ \iota$. Therefore, $\iota'\circ\alpha\circ p\in \bfJ\cap \bfR_{\L}^\leq=\bfI$. Hence $\alpha = p'\circ\iota'\circ\alpha\circ p\circ \iota \in \bfI'$, which is a contradiction. This completes the proof.\endproof

\begin{lemma}\label{goingdown}
Assume $\L'$ is weak left cell and $\L$ a strong left cell contained in $\L'$. Let $\bfI\in \MaxSpec(\bfR_{\L'}^\preceq)$. Then $\bfI\cap \bfR_{\L}^\leq\in  \MaxSpec(\bfR_{\L}^\leq)$.
\end{lemma}

\proof
We first show that $\bfI\cap \bfR_{\L}^\leq\in \Spec(\bfR_{\L}^\leq)$. Assume, for a contradiction, that it contains $\id_X$ for any $X\in\L$. Since $X\in \L'$, $\bfI$ contains the ideal generated by $\id_X$ in $\bfR_{\L'}^\preceq$, which equals $\bfR_{\L'}^\preceq$. This contradicts $\bfI\in \Spec(\bfR_{\L'}^\preceq)$.

We now claim that $\bfI\cap \bfR_{\L}^\leq$ is maximal. Assume $\bfJ\in \Spec(\bfR_{\L}^\leq)$ strictly containing $\bfI\cap  \bfR_{\L}^\leq$, so that there exists $\alpha\in \bfJ(\tj)$ for some $\tj\in \cC$, but $\alpha \notin \bfI(\tj)$. Let $\bfJ'$ be the ideal in $\bfR_{\L'}^\preceq$ generated by $\bfJ$ and note that, by Lemma \ref{downup}, $\bfJ'\supseteq \bfI$.
Then $\alpha \in \bfJ'(\tj)$ but $\alpha \notin \bfI(\tj)$. 
Maximality  of $\bfI$ implies $\bfJ'=\bfR_{\L'}^\preceq$, so, by Lemma \ref{downup}, $\bfJ=\bfJ'\cap \bfR_{\L}^\leq= \bfR_{\L}^\leq$ contradicting $\bfJ\in \Spec(\bfR_{\L}^\leq)$. This completes the proof.
\endproof

\begin{theorem}\label{maxspecbij}
If $\L$ is a strong left cell contained in a weak left cell $\L'$, then there is a bijection between $\MaxSpec(\bfR_{\L}^\leq)$ and $\MaxSpec(\bfR_{\L'}^\preceq)$, sending an ideal $\bfI\in\MaxSpec(\bfR_{\L}^\leq)$ to the ideal generated by $\bfI$ in $\bfR_{\L'}^\preceq$, and an ideal $\bfJ\in \MaxSpec(\bfR_{\L'}^\preceq)$ to $\bfJ\cap \bfR_{\L}^\leq$.
\end{theorem}

\proof
Let $\bfI\in\MaxSpec(\bfR_{\L}^\leq)$. Then the ideal $\bfI'$ generated by $\bfI$ in $\bfR_{\L'}^\preceq$ is in $\MaxSpec(\bfR_{\L'}^\preceq)$ by Lemma \ref{goingup}. 
Conversely, given $\bfJ\in \MaxSpec(\bfR_{\L'}^\preceq)$, Lemma \ref{goingdown} shows that $\bfJ\cap \bfR_{\L}^\leq$ is in $\MaxSpec(\bfR_{\L}^\leq)$. 

Since  $\bfI'\cap \bfR_{\L}^\leq=\bfI$ (see Lemma \ref{idealofsub}), and the ideal generated by $\bfJ\cap \bfR_{\L}^\leq$ inside $\bfR_{\L'}^\preceq$ is $\bfJ$ (see Lemma \ref{downup}), these two constructions are mutual inverses.
\endproof

\begin{lemma}\label{maxrestrict}
Let $\L,\L'$ be weak left cells such that $\L\preceq_L\L'$. Let $\bfI\in \MaxSpec(\bfR_{\L}^\preceq)$. Then the restriction of $\bfI$ to $\bfR^\preceq_{\L'}$ is either in $\MaxSpec(\bfR_{\L'}^\preceq)$ or equal to $\bfR^\preceq_{\L'}$.
\end{lemma}

\proof
For $\bfJ\in \Spec(\bfR_{\L}^\preceq)$, we write $\bfJ_{\L'}$ for its restriction to $\bfR^\preceq_{\L'}$ and, for
 $\bfJ\in \Spec(\bfR_{\L'}^\preceq)$, we write $\bfJ^{\L}$ for the ideal generated by $\bfJ$ in $\bfR^\preceq_{\L}$.
 

 Assume $\bfI_{\L'}$ is not maximal and choose $\bfJ\in \MaxSpec(\bfR_{\L'}^\preceq)$ with $\bfI_{\L'}\subsetneq \bfJ$. This implies that there exists a morphism $f$ in some $\bfJ(\tj)$ with $f\notin \bfI(\tj)$. Hence $\bfI\subsetneq \bfJ^\L +\bfI$. By maximality of $\bfI$, $\bfJ^\L +\bfI$ is equal to $\bfR^\preceq_{\L}$.
 
So $(\bfJ^\L +\bfI)_{\L'}=(\bfJ^\L)_{\L'} +\bfI_{\L'}$ equals $\bfR^\preceq_{\L'}$. 

Moreover, since $(\bfJ^\L)_{\L'}\supseteq \bfJ \supsetneq \bfI_{\L'}$, we see that $(\bfJ^\L +\bfI)_{\L'}=(\bfJ^\L)_{\L'} = \bfJ$ where the last equality follows form Lemma \ref{easyideals} \eqref{idealtrivial}. 

Thus $\bfJ = \bfR^\preceq_{\L'}$, which contradicts $\bfJ\in \MaxSpec(\bfR_{\L'}^\preceq)$.
\endproof

\subsection{Construction of strong and weak cell \texorpdfstring{$2$}{2}-representations}\label{cell2section}
We now introduce weak and strong dg cell $2$-representations, the main constructions of this article.

\begin{definition}\label{def:cell2rep}
For $\bfI$ in $\MaxSpec(\bfR_{\L}^\leq)$ (respectively, $\MaxSpec(\bfR_{\L}^\preceq)$), we define the {\bf strong} (respectively, {\bf weak}) {\bf cell $2$-representation} $\bfC_\bfI^\leq$  (respectively, $\bfC_\bfI^\preceq$) to be the quotient of $\bfR_\L^\leq$ (respectively, $\bfR_\L^\preceq$) by $\bfI$. 
\end{definition}

\begin{example}\label{geomex} Recall the setup  of Example \ref{geomex1}. 
By construction, there is a unique weak left, right and two-sided cell $\L$, namely that containing $\one$. Moreover, $\bfR_{\L}^\preceq$ is simply the unique principal pretriangulated $2$-representation $\bfP$. Then $\Spec(\bfR_{\L}^\preceq)$ (respectively, $\MaxSpec(\bfR_{\L}^\preceq)$) is given by the usual (maximal) spectrum $\Spec(R)$ (respectively,  $\MaxSpec(R)$), recovering the affine scheme associated to $R$.
The category underlying the cell $2$-representation corresponding to $\mathbf{m}\in \MaxSpec(R)$ is given by bounded complexes of free modules over $R/\mathbf{m}$.
\end{example}

\begin{proposition}\label{cellinclude}
Let $\L$ be a strong left cell contained in a weak left cell $\L'$, $\bfI\in \MaxSpec(\bfR_{\L}^\leq)$ and $\bfI'$ the ideal generated by $\bfI$ in $\bfR_{\L'}^\preceq$, which is in $\MaxSpec(\bfR_{\L'}^\preceq)$ by Lemma \ref{goingup}. Then $\bfC_\bfI^\preceq\in \cC\tworep$ and $\bfC_\bfI^\leq$ is a pretriangulated $2$-subrepresentation of $\bfC_{\bfI'}^\preceq$.
\end{proposition}
\proof
We know that $\bfR_\L^\leq$ is a pretriangulated $2$-subrepresentation of $\bfR_{\L'}^\preceq\in \cC\tworep$. Since $\bfI\subseteq \bfI'$, we obtain an induced morphism of pretriangulated $2$-representations
$$\Phi_\bfI\colon \bfC_\bfI^\leq=\bfR_\L^\leq/\bfI\longrightarrow \bfC_{\bfI'}^\preceq=\bfR_{\L'}^\preceq/{\bfI'}.$$
By definition, $\bfC_\bfI^\leq$ and $\bfC_{\bfI'}^\preceq$ are pretriangulated.
Since $\bfI'\cap \bfR_\L^\leq=\bfI$ by Theorem \ref{maxspecbij}, it follows from Lemma \ref{quotientfaithful} that $\Phi_\bfI$ is fully faithful.  
This concludes the proof.
\endproof

Furthermore, by Lemma \ref{easyideals}\eqref{quotientpretri1}, $\bfC_\bfI^\leq$ and $\bfC_{\bfI'}^\preceq$ are quotient-simple pretriangulated $2$-representations of $\cC$, and  $\bfC_\bfI^\leq$  is, in addition, cyclic, while $\bfC_{\bfI'}^\preceq$ is only weakly $\cC$-generated by any $1$-morphism contained in the associated cell.

\begin{corollary}
Let $\L$ be a strong left cell inside a weak left cell $\L'$. Let $\bfI'\in\MaxSpec(\bfR_{\L'}^\preceq)$  such that $\bfC_{\bfI'}^\preceq$ is acyclic. Let $\bfI\in\MaxSpec(\bfR_{\L}^\leq)$ be the ideal associated to $\bfI'$ by Lemma \ref{goingdown}. Then $\bfC^{\leq}_{\bfI}$ is acyclic.
\end{corollary}

\proof
Since $\bfC^{\leq}_{\bfI}$ is a dg $2$-subrepresentation of $\bfC^\preceq_{\bfI'}$, the claim follows from Proposition \ref{cellinclude}.
\endproof

We say that a (weak or strong) left cell $\L$ is {\bf acyclic} if for any $1$-morphism $\rX$ of $\L$ there exists an endomorphism $f$ of $\rX$ such that $\del(f)=\id_\rX$. The following observation is almost immediate.

\begin{lemma}
Let $\L$ be an acyclic strong (resp. weak) left cell and $\bfI\in \MaxSpec(\bfR_{\L}^\leq)$ (resp. $\bfI\in \MaxSpec(\bfR_{\L}^\preceq)$). Then $\bfC^{\leq}_{\bfI}$ (resp. $\bfC^\preceq_{\bfI}$) is acyclic.
\end{lemma}

\proof
Since $\L$ is acyclic, there exists a $2$-morphism $f_\rX$ with $\del(f_\rX)=\id_\rX$ for every $\rX\in\L$. As $\bfI$ is a proper ideal, it cannot contain $\id_\rX$ for any $\rX\in \L$ and since it is a dg ideal, it can therefore also not contain $f_\rX$ for any $\rX\in\L$. Hence in the cell $2$-representation $\del(f_{\rX}+\bfI) = \id_\rX+\bfI$ and the cell $2$-representation is acyclic. 
\endproof

\subsection{Cell \texorpdfstring{$2$}{2}-representations and dg idempotent completion}\label{secidempcomp}

Let $\cC$ be a dg $2$-category and $\bfM\in \cC\tworep$. Consider the dg idempotent completion $\cC^\circ$.

Define $\bfM^\circ$ to be the pretriangulated $2$-representation of $\cC$ on the categories $\bfM(\ti)^\circ$. Then $\bfM^\circ\in \cC\tworep$ using Lemma \ref{Ccirctri}. Further, $\bfM^\circ$ has a natural extension to a pretriangulated $2$-representation of $\cC^\circ$ by Lemma~\ref{extendtocirc}.

\begin{lemma}
The pretriangulated $2$-representations $\big(\bfP^{\cC}_\ti\big)^\circ$ and $\bfP^{\cC^\circ}_\ti$ of $\cC^\circ$ are dg equivalent.
\end{lemma}
\begin{proof}
For any object $\tj$ of $\cC$, $\big(\bfP^{\cC}_\ti\big)^\circ(\tj)=\cC(\ti,\tj)^\circ\simeq \ov{\cC(\ti,\tj)^\circ}=\bfP^{\cC^\circ}_\ti(\tj)$ are dg equivalent pretriangulated categories by Lemma~\ref{Ccirctri}, with the same action of $\cC^\circ$.
\end{proof}

\begin{lemma}\label{RinRcirc} Assume that $\rX$ is a dg indecomposable $1$-morphism in $\cC$.
\begin{enumerate}
\item\label{RinRcirc1} Let  $\rY$ be a dg direct summand of $\rX$ in $\cC^\circ$. Then the $2$-representation $\bfR_\rY^{\cC^\circ,\leq}$ is a pretriangulated $2$-subrepresentation of $\big(\bfR_\rX^{\cC,\leq}\big)^\circ$ of $\cC^\circ$. 
Moreover,  $\bfR_\rX^{\cC^\circ,\leq} =\big(\bfR_\rX^{\cC,\leq}\big)^\circ$.
\item\label{RinRcirc2} The $2$-representation $\big(\bfR_\rX^{\cC,\preceq}\big)^\circ$ is a pretriangulated $2$-subrepresentation of  $\bfR_\rX^{\cC^\circ,\preceq}$. 
\end{enumerate}
\end{lemma}
\begin{proof}
Part \eqref{RinRcirc1}. Let $\rX\in \cC(\ti,\tj)$.
For any object $\tk$ of $\cC$, $\bfR_\rY^{\cC^\circ,\leq}(\tk)$ is the thick closure of the set $\{\rG \rY \,\vert\, \rG\in \cC(\tj,\tk)^\circ\}$ in $\cC(\tj,\tk)^\circ$.
Note that $\big(\bfR_\rX^{\cC,\leq}\big)^\circ(\tk)$ is the dg idempotent completion of the thick closure of the set $\{\rF \rX \,\vert\, \rF\in \cC(\tj,\tk)\}$ in $\cC(\tj,\tk)^\circ$. Hence, 
$$\{\rG \rY \,\vert\, \rG\in \cC(\tj,\tk)^\circ\}\subset  \big(\bfR_\rX^{\cC,\leq}\big)^\circ(\tk).$$
The larger category is closed under shifts, cones and dg direct summands, and hence also contains $\bfR_\rY^{\cC^\circ,\leq}(\tk)$.

Conversely, thick closure of the set $\{\rF \rX \,\vert\, \rF\in \cC(\tj,\tk)\}$ in $\cC(\tj,\tk)^\circ$ is clearly contained in the thick closure of the set $\{\rG \rX \,\vert\, \rG\in \cC(\tj,\tk)^\circ\}$ in $\cC(\tj,\tk)^\circ$.

Part \eqref{RinRcirc2}. As categories, $\bfR_{\rX}^{\cC,\preceq}(\tk)\subseteq \bfR_{\rX}^{\cC^\circ,\preceq}(\tk)$. The larger one is a subcategory of $\cC^\circ(\ti,\tk)$ and is clearly dg idempotent complete, hence contains the dg idempotent completion $\big(\bfR_\rX^{\cC,\preceq}\big)^\circ$ of the smaller one.
\end{proof}

To compare dg cell $2$-representations for $\cC$ and $\cC^\circ$, we need the following general observation.

\begin{lemma}\label{Icircquotient}
Let $\bfM$ be a pretriangulated $2$-representation of $\cC$ and $\bfI$ a dg ideal in $\bfM$ with quotient morphism $\Pi\colon \bfM\to \bfM/\bfI$. Then $\Pi^\circ$ induces a fully faithful morphism of pretriangulated $2$-representations from $\bfM^\circ/\bfI^\circ$ to $\big(\bfM/\bfI\big)^\circ$.  Moreover, this morphism induces an equivalence between $(\bfM^\circ/\bfI^\circ)^\circ$ and $\big(\bfM/\bfI\big)^\circ$.
\end{lemma}
\begin{proof}
The first statement follows from Lemma \ref{quotientfaithful} using that $\bfI=\bfM\cap \bfI^\circ$ by Lemma \ref{circideal}.
To see that the induced morphism between dg idempotent completions is essentially surjective, it suffices to observe that every dg idempotent in some $\bfM/\bfI(\ti)$ is also a dg idempotent in $\bfM^\circ/\bfI^\circ (\ti)$, and hence the associated object is in the essential image of the morphism when extended to  $(\bfM^\circ/\bfI^\circ)^\circ$.
\end{proof}

\begin{lemma}\label{bfIcirc}
Given $\bfI\in \MaxSpec(\bfR^{\cC,\leq}_{\rX})$ (respectively, $\bfI\in  \MaxSpec(\bfR^{\cC,\preceq}_{\rX})$), the ideal $\bfI^\circ$ generated by $\bfI$ in $(\bfR_{\rX}^{\leq})^{\circ}$ (respectively, $(\bfR_{\rX}^{\preceq})^{\circ}$) is also maximal.
\end{lemma}
\begin{proof}
Let $\bfI\in \MaxSpec(\bfR^{\cC,\leq}_{\rX})$ so that $\id_\rX\notin \bfI(\ti,\tj)$. Then $\id_\rX\notin \bfI^\circ(\ti,\tj)$ so $\bfI$ is a proper ideal. Now $\bfI^\circ$ is maximal using Lemma \ref{circideal}. The same argument applies using $\bfR^{\preceq}$ instead of $\bfR^{\leq}$.
\end{proof}

Note that a $1$-morphism $\rX\in \cC(\ti,\tj)$ remains dg indecomposable in $\cC^\circ(\ti,\tj)$ if and only if $\id_\rX$ and $0$ are its only dg idempotent endomorphisms. In this case, Lemma \ref{RinRcirc} shows that we have $\bfR_{\rX}^{\cC^\circ,\leq}= (\bfR_{\rX}^{\cC,\leq})^{\circ}$ and that any $\bfI\in \MaxSpec(\bfR^{\cC,\leq}_{\rX})$ gives rise to $\bfI^\circ\in \MaxSpec(\bfR^{\cC^\circ,\leq}_{\rX}).$ Denote the corresponding cell $2$-representations by $\bfC_\bfI^{\cC, \leq}$ and  $\bfC_{\bfI^\circ}^{\cC^\circ, \leq}$, respectively, to clearly identify which $2$-category the $2$-representation is defined on.

\begin{proposition}
Assume $\rX\in \cC(\ti,\tj)$ remains dg indecomposable in $\cC^\circ(\ti, \tj)$ and let $\bfI\in \MaxSpec(\bfR^{\cC,\leq}_{\rX})$. Then we have a fully faithful morphism of pretriangulated $2$-representation
 $$\bfC_{\bfI^\circ}^{\cC^\circ, \leq} \hookrightarrow \big(\bfC_\bfI^{\cC, \leq}\big)^\circ,$$
which induces an equivalence between $\big(\bfC_{\bfI^\circ}^{\cC^\circ, \leq}\big)^\circ$ and $\big(\bfC_\bfI^{\cC, \leq}\big)^\circ$.
\end{proposition}

\proof
This follows immediately from Lemma \ref{Icircquotient}.
\endproof

We next consider weak cell $2$-representations.

\begin{lemma}\label{weakbfImax}
Assume $\rX\in \cC(\ti,\tj)$ remains dg indecomposable in $\cC^\circ(\ti, \tj)$ and let $\bfI\in \MaxSpec(\bfR^{\cC,\preceq}_{\rX})$. Let $\bfJ$ be the ideal generated by $\bfI$ in $\bfR^{\cC^\circ,\preceq}_{\rX}$. Then $\bfJ\in \MaxSpec(\bfR^{\cC^\circ,\preceq}_{\rX})$.
\end{lemma}

\proof
Since $\rX$ weakly $\cC$-generates $\bfR^{\cC^\circ,\preceq}_{\rX}$, Lemma \ref{2repidealonoverline} implies that $\bfJ$ is completely determined by the spaces $\Hom_{\bfJ(\tk)}(\rF_e\rX, \rG_{e'}\rX)$ where $\rF,\rG$ run over dg indecomposable $1$-morphisms in $\cC$, and $e,e'$ run over their dg idempotent endomorphisms. These spaces, however, are uniquely determined by the space $\Hom_{\bfJ(\tk)}(\rF\rX, \rG\rX)$, which equals $\Hom_{\bfI(\tk)}(\rF\rX, \rG\rX)$ by Lemma \ref{idealofsub}. In particular, the ideal in $\bfR^{\cC^\circ,\preceq}_{\rX}$ generated by $\big(\bfR^{\cC,\preceq}_{\rX}\big)^\circ$ and hence the ideal generated by all identities on objects in any $\bfR^{\cC,\preceq}_{\rX}(\tk)$ is all of $\bfR^{\cC^\circ,\preceq}_{\rX}$.

Let $f\in \Hom_{\bfR^{\cC^\circ,\preceq}_{\rX}(\tk)}(M, N)$, for some $M,N\in \bfR^{\cC^\circ,\preceq}_{\rX}(\tk)$, and assume $f\notin \bfJ(\tk)$. Consider the ideal $\bfJ'$  of $\bfR^{\cC^\circ,\preceq}_{\rX}$ generated by $\bfJ$ and $f$, or equivalently by $\bfI$ and $f$. We need to show $\bfJ' = \bfR^{\cC^\circ,\preceq}_{\rX}$. 
 Writing $\id_M = p_M \id_{\rF_e\rX} \iota_M, \id_N = p_N\id_{\rG_{e'}\rX}\iota_N$, and $e=\iota   p, e'=\iota' p'$,  we have
$f'=(\iota'\circ_0\id_\rX) \iota_N f p_M(p\circ_0 \id_\rX) \in \Hom_{\bfJ'(\tk)}(\rF\rX, \rG\rX)$. If $f'\in  \Hom_{\bfI(\tk)}(\rF\rX, \rG\rX)$, then
$$f= p_N(p'\circ_0 \id_\rX)  f' (\iota\circ_0\id_\rX) \iota_M\in \Hom_{\bfJ(\tk)}(M, N),$$ a contradiction. Thus, $f'\notin \bfI(\tk)$, so the ideal $\bfI'$ generated by $\bfI$ and $f'$ in $\bfR^{\cC,\preceq}_{\rX}$ equals $\bfR^{\cC,\preceq}_{\rX}$ by maximality of $\bfI$. This implies that the ideal $\bfJ''$ generated by $\bfI$ and $f'$ in $\bfR^{\cC^\circ,\preceq}_{\rX}$ equals $\bfR^{\cC^\circ,\preceq}_{\rX}$ by the first paragraph.
Since $f'\in \bfJ'(\tk)$ and $\bfI\subseteq \bfJ'$, the ideal  $\bfJ''$ is contained in $\bfJ'$, so $\bfJ' =\bfR^{\cC^\circ,\preceq}_{\rX}$, as required.
\endproof

\begin{lemma}
With the notation of Lemma \ref{weakbfImax}, $\big(\bfC^{\cC,\preceq}_\bfI\big)^\circ$ is a pretriangulated $2$-subrepresentations of $\big(\bfC_\bfJ^{\cC^\circ,\preceq}\big)^\circ$.
\end{lemma}

\begin{proof}
The kernel of the composition $\bfR^{\cC,\preceq}_{\rX} \hookrightarrow \bfR^{\cC^\circ,\preceq}_{\rX} \to \bfC_\bfJ^{\cC^\circ,\preceq}$ equals $\bfI$ and hence induces a fully faithful morphism of pretriangulated $2$-representations
$$\bfC^{\cC,\preceq}_\bfI=\bfR^{\cC,\preceq}_{\rX}/\bfI \hookrightarrow\bfC_\bfJ^{\cC^\circ,\preceq}$$
of $\cC$. Taking dg idempotent completions, we obtain a fully faithful morphism of pretriangulated $2$-representations of $\cC^\circ$ as claimed.
\end{proof}

\begin{corollary} \label{bijforindec}
Assume $\rX\in \cC(\ti,\tj)$ remains dg indecomposable in $\cC^\circ(\ti, \tj)$. Then we have bijections
\begin{gather*} 
\MaxSpec(\bfR^{\cC,\leq}_{\rX})\stackrel{\sim}{\longleftrightarrow}\MaxSpec(\bfR^{\cC^\circ,\leq}_{\rX}),\\
\MaxSpec(\bfR^{\cC,\preceq}_{\rX})\stackrel{\sim}{\longleftrightarrow}\MaxSpec(\bfR^{\cC^\circ,\preceq}_{\rX}).
\end{gather*}
\end{corollary}

\proof
The statements follows immediately from combining Lemma \ref{RinRcirc}\eqref{RinRcirc1} with Lemma \ref{bfIcirc} for $\leq$, and from Lemma \ref{weakbfImax} for $\preceq$, in both cases using Lemmas \ref{idealofsub} and \ref{LemIcapM}.
\endproof

\begin{lemma}\label{specmapscirc}
Let $\rX\in \cC(\ti,\tj)$ be dg indecomposable, $e\colon \rX\to \rX$ a dg idempotent and set $\rY=\rX_e \in \cC^\circ(\ti,\tj)$. For $\bfI\in \Spec(\bfR_\rX^\leq)$ (respectively, $\bfI\in \Spec(\bfR_\rX^\preceq)$), we define 
$\bfI_e= \bfI^\circ \cap \bfR_{\rY}^{\cC^\circ,\leq}$ (respectively, $\bfI_e= \bfI^\circ \cap \bfR_{\rY}^{\cC^\circ,\preceq}$). 
\begin{enumerate}[(a)]
\item\label{specmap1} If $\bfI\in \Spec(\bfR_\rX^\leq)$ does not contain $e$, then $\bfI_e\in \Spec(\bfR_{\rY}^{\cC^\circ,\leq})$. 
\item\label{specmap2} If $\bfI\in \MaxSpec(\bfR_\rX^\leq)$ does not contain $e$, then $\bfI_e\in \MaxSpec(\bfR_{\rY}^{\cC^\circ,\leq})$.  
\end{enumerate}
The same statements hold replacing $\leq$ with $\preceq$.
\end{lemma}
\begin{proof}
We will prove the statements in the case of $\preceq$, the statements for $\leq$ are proved analogously,
using the observation that dg ideals in $\bfR_{\rX}^{\leq}$ and $\bfR_{\rY}^{\leq}$ are also determined by morphisms on objects of the form $\rG \rX$, respectively, $\rG\rY$ for $1$-morphisms $\rG$ in $\cC$, see Lemma \ref{2repidealonoverline}.
\begin{enumerate}[(a)]
\item Lemma \ref{2repidealonoverline} implies that  $\bfI_e$ is  generated as an ideal by morphisms of the form 
$$(\id_\rG \circ_0 e)f(\id_\rF \circ_0 e), \qquad \text{for }f\in \Hom_{\bfI(\tk)}(\rF\rX,\rG\rX), \quad \rF,\rG\in \cC(\tj,\tk).$$
This description implies that if $e\notin \bfI(\tj)$, then also $e\notin \bfI_e(\tj)$.

\item Suppose that $g\in \bfR_\rY^{\cC^\circ,\preceq}(\tk)$ is not contained in $\bfI_e(\tk)$. Without loss of generality, $g\colon \rF\rY\to \rG\rY$ for some $\rF,\rG \in \cC(\tj,\tk)$. 
Note that $g$ corresponds to a morphism
$$g=(\id_\rG \circ_0 e)g(\id_\rF \circ_0 e)\colon \rF \rX\to \rG\rX.$$ This morphism $g$ cannot be in $\bfI$. By maximality of $\bfI$, we find 
$$\id_{\rX}=h+ \sum_i a_i\circ g\circ b_i+\sum_j c_j\circ \del(g)\circ d_j,$$ for some morphisms $a_i,b_i,c_j,d_j$ and $h\in \bfI(\tk)$. Thus, 
$$\id_{\rY}=e\circ \id_{\rX}\circ e=e\circ h\circ e+ \sum_i e\circ a_i\circ g\circ b_i\circ e+\sum_j e\circ c_j\circ \del(g)\circ d_j\circ e.$$
This shows that the dg ideal in $\bfR_\rY^{\cC^\circ,\preceq}$ generated by $\bfI_e$ and $g$ contains all morphisms, and hence $\bfI_e$ is maximal.\qedhere
\end{enumerate}
\end{proof}

Using Proposition \ref{specmapscirc}, we obtain the following inclusions of quotient-simple pretriangulated $2$-representations.

\begin{proposition}\label{idempcellincl}
Let $\rX\in \cC(\ti,\tj)$ be dg indecomposable, $e\colon \rX\to \rX$ a dg idempotent and set $\rY=\rX_e\in \cC(\ti,\tj)^\circ$. Given $\bfI\in \MaxSpec(\bfR_{\rX}^{\leq})$, respectively, $\bfI\in \MaxSpec(\bfR_{\rX}^{\preceq})$, such that $e$ is not in $\bfI$, we obtain fully faithful morphisms of quotient-simple pretriangulated $2$-representations
\begin{gather*}
(\bfC^{\leq}_{\bfI_e})^\circ \hookrightarrow  (\bfC^{\leq}_{\bfI})^\circ \qquad \text{and} \qquad 
 (\bfC^{\preceq}_{\bfI_e})^\circ\hookrightarrow (\bfC^{\preceq}_{\bfI})^\circ.
\end{gather*}
\end{proposition}
\begin{proof}
The claimed morphisms of pretriangulated $2$-representations are given by the following compositions.
\begin{gather*}
\bfC^{\leq}_{\bfI_e}=\bfR_{\rY}^{\leq}/\bfI_e\hookrightarrow (\bfR_{\rX}^{\leq})^{\circ}/\bfI^{\circ}\hookrightarrow (\bfR_{\rX}^{\leq}/\bfI)^{\circ}=(\bfC^{\leq}_{\bfI})^\circ, \\
 \bfC^{\preceq}_{\bfI_e}=\bfR_{\rY}^{\preceq}/\bfI_e\hookrightarrow (\bfR_{\rX}^{\preceq})^{\circ}/\bfI^{\circ}\hookrightarrow (\bfR_{\rX}^{\preceq}/\bfI)^{\circ}=(\bfC^{\preceq}_{\bfI})^\circ.
\end{gather*}
The first morphisms in these compositions are fully faithful by Lemma \ref{quotientfaithful}, using that $\bfI_e= \bfI^\circ \cap \bfR_{\rY}^{\cC^\circ,\leq}$ (respectively, $\bfI_e= \bfI^\circ \cap \bfR_{\rY}^{\cC^\circ,\preceq}$) by definition. The second morphisms are fully faithful by Lemma \ref{Icircquotient}. All pretriangulated $2$-representations involved in the above inclusions are quotient-simple by Lemma \ref{lem-simpletrans-circ}.
\end{proof}

\begin{remark}
We point out that at least in the case of $\leq$, the fully faithful morphism of pretriangulated $2$-representations in Proposition \ref{idempcellincl} is not an equivalence as the following trivial example shows. Let $\cC$ be the one-object dg $2$-category with endomorphism category $\Bbbk\dgmod$ and note that $\cC^\circ = \cC$.  Consider $X=\one \oplus (\one\oplus \one, \mat{ 0 &\id\\0&0 })$ and $\rY=(\one\oplus \one, \mat{ 0 &\id\\0&0 })$. Then $\bfR_{\rY}^{\leq}$ is quotient-simple and hence equal to the unique dg cell $2$-representation associated to its cell, which is acyclic.
On the other hand $\bfR_{\rX}^{\leq} = (\bfR_{\rX}^{\leq})^{\circ} = \bfP$, which is also quotient-simple, but not acyclic as the underlying category contains the object $\one$.
\end{remark}

\subsection{Local endomorphism rings}\label{locendosec}

In this subsection, we assume that every $\cC(\ti,\tj)$ is dg equivalent to the thick closure of a set $\left\{ \rF_s | s\in \tI_{\ti,\tj} \right\}$, where each $\rF_s$ has a local endomorphism ring (as a $\Bbbk$-algebra).

Note that with $\cC$, the dg idempotent completion $\cC^\circ$ also satisfies this condition. In this subsection, we provide a bijection between the weak cell $2$-representations of $\cC$ and $\cC^\circ$ in this setup.

\begin{lemma}\label{overlineincell}
Every weak left cell contains a $1$-morphism of the form $\rY=\left(\bigoplus_{i=1}^s \rF_i, \alpha\right)$ for $1$-morphisms $\rF_i$ with local endomorphism ring.
\end{lemma}

\proof
Let $\rX\in \cC(\ti,\tk)$ be a dg indecomposable $1$-morphism and $\L$ its left cell.  Then we can find $\rX'=\left(\bigoplus_{i=1}^s \rF_i, \alpha' \right)\in \cC(\ti,\tk)$ such that $\rX$ is a dg direct summand of $\rX'$ via dg morphisms $p\colon \rX'\to \rX$ and $\iota\colon \rX\to \rX'$. 

Consider the morphisms $p_i\iota p\iota_j\in \Hom_{\cC(\ti,\tk)}(\rF_i, \rF_j)$. Since  $\sum_{i,j=1}^s \iota_ip_i\iota p\iota_jp_j =\id_{\rX}$, not all $p_i\iota p\iota_j$ can be radical morphisms, so there exist some $i,j$ for which these are isomorphisms (see e.g. \cite[A.3, Proposition 3.5(b)]{ASS}). For these, we can compose with the inverses to obtain $\rX \preceq_\L \rF_i$ and $\rF_i\in \bfR_{\L'}^\preceq(\tk)$, and similarly for $\rF_j$. Let $\tI$ be the set of $i\in \lbrace 1,\ldots, s \rbrace$ such that there exists a $j$ with
$p_i\iota p\iota_j$ invertible, and hence $\rF_i\in \bfR_{\L'}^\preceq(\tk)$.

Notice that, by construction, $$\id_\rX = \underbrace{\sum_{j\in \tI} p  \iota_j p_j \iota}_a +\underbrace{\sum_{j\notin \tI} p  \iota_j p_j \iota .}_b$$ We claim that $b\in \rad\End_{\cC(\ti,\tk)} (\rX)$.

Indeed, by definition, $j\notin \tI$ means that, for all $k\in \{1,\dots, s\}$,  the composition $p_j\iota p \iota_k$ is not invertible, and hence  in $\rad\Hom_{\cC(\ti,\tk)}(\rF_k,\rF_j)$ (\cite[A.3, Proposition~3.5(b)]{ASS}). Thus $\sum_{k=1}^sp_j\iota p \iota_kp_k  = p_j\iota p \in \rad\Hom_{\cC(\ti,\tk)}(\rX',\rF_j)$ and, therefore, $p i_j p_j\iota p \iota = p i_j p_j\iota \in \rad\End_{\cC(\ti,\tk)}(\rX)$. We conclude that $b$ is a sum of radical elements and hence in the radical itself.

Thus $a$ is invertible and $\rY\preceq_L \rX$, where $\rY = (\bigoplus_{i\in \tI} \rF_i, \alpha)$ for $\alpha_\rY$ the submatrix of $\alpha'$ corresponding to $\tI$. Further, $\rX\preceq_L \rF_i$ for all $i\in \tI$ implies that $\rX\preceq_L \rY$. Hence, $\rX$ and $\rY$ are equivalent in the weak left order. 
\endproof

\begin{lemma}\label{biggercellskill1}
Let $\bfC$ be a weak dg cell $2$-representation of $\cC$. Then there is a left cell $\L$ containing one of the $1$-morphisms $\rF_i$ with local endomorphism ring and some $\bfI \in \MaxSpec(\bfR_{\L}^\preceq)$ such that $\bfC$ is dg equivalent to $\bfC_\bfI^\preceq$. 
\end{lemma}
\proof

Let $\L'$ be a left cell of $\cC$ and $\bfI'\in \MaxSpec(\bfR_{\L'}^\preceq)$ such that $\bfC = \bfC_{\bfI'}^\preceq$. Let $\rX\in\L' \cap \cC(\ti,\tk)$ for some $\ti,\tk\in \cC$. By Lemma \ref{overlineincell}, we may assume that $\rX=\left(\bigoplus_{i=1}^s \rF_i, \alpha\right)$, so in particular, $\rF_i\in \bfR_{\L'}^\preceq(\tk)$. Consider the (non-dg) projection $p_i\colon \rX'\to \rF_i$ and injection $\iota_i\colon\rF_i\to \rX'$. Then
$\id_\rX =\sum_{i=1}^s \iota_ip_i$ and since $\id_\rX\notin \bfI'(\tk)$, we see that not all $\id_{\rF_i}$ can be in $ \bfI'(\tk)$. Consider the (non-empty) subset 
$$\tJ=\{i\in\{1, \ldots, s\}\,|\, \id_{\rF_i}\notin \bfI'(\tk)\}.$$

Let $\rF_t$ be maximal with respect to $\preceq_L$ among the $\rF_j$ with $j\in \tJ$. We claim that maximality of $\bfI'$ implies that $\rF_t$ is also \emph{minimal} among the $\rF_j$ with $j\in\tJ$ in the weak left order. Indeed, consider the ideal $\bfI''$ generated by $\bfI'$ and $\id_{\rF_t}$, which clearly strictly contains $\bfI'$. By maximality of $\bfI'$, this implies $\bfI'' = \bfR_{\L'}^\preceq$ and, for any $j\in \tJ$, using $\rF_{j}\in \bfR_{\L'}^\preceq(\tk)$ as shown above, we can write $\id_{\rF_j}=f\circ \id_{\rH\rF_t}\circ g +r$ for some $1$-morphism $\rH$, some $2$-morphisms $f,g$ in $\cC$ and $r\in\bfI'(\tk)$.

Notice that $r\in\bfI'(\tk)$ implies that  $r$ is not a unit for $j\in \tJ$, and hence contained in the unique maximal ideal of $\End_{\cC(\ti,\tk)}(\rF_j)$. Thus $\id_{\rF_j}-r$ is invertible, so $f\circ \id_{\rH\rF_t}\circ g$ is invertible and $\rF_j\succeq_L \rF_t$. 
Thus $\rF_t$ is indeed also minimal among the $\rF_j$ with $j\in \tJ$ in the weak left order. Hence, all $\rF_j$ with $j\in \tJ$ are in the same weak left cell.

Denote the cell containing the $\rF_j$ with $j\in \tJ$ by $\L$ and observe that $\bfR_{\L}^\preceq \subseteq\bfR_{\L'}^\preceq$ by Lemma~\ref{indep}\eqref{indep2}, using that $\rX\preceq_L \rF_t$. 

Define $\bfI$ to be the restriction to $\bfR_{\L}^\preceq$ of $\bfI'$. We then obtain a dg morphism of cell $2$-representations $\bfC_\bfI \to \bfC_\bfI'$ induced by the inclusion of $\bfR_{\L}^\preceq$ into $\bfR_{\L'}^\preceq$, which is hence fully faithful. We claim that it is also essentially surjective.

By the considerations above, 
$\id_\rX = \sum_{j\in \tJ} \iota_jp_j + \bfI'(\tk),$ and hence $\rX$ is isomorphic in $\bfC_{\bfI'}(\tk)$ to $(\bigoplus_{j\in \tJ} \rF_j, \alpha_\tJ)$ for $\alpha_\tJ$ the submatrix of $\alpha$ corresponding to $\tJ$, which is in the image of the inclusion. Since $\rX$ weakly $\cC$-generates $\bfC_{\bfI'}$, this completes the proof.
\endproof

\begin{corollary}\label{cor:cell2rep-circ}
Under the assumptions on $\cC$ of this section, there is a bijection between the sets of equivalence classes of weak dg cell $2$-representations of $\cC$ and those of $\cC^\circ$.
\end{corollary}

\begin{proof}
Lemma \ref{biggercellskill1} implies that the dg cell $2$-representations of $\cC$ and $\cC^\circ$ are parametrized by $\MaxSpec(\bfR_{\rF_i}^{\cC,\preceq})$ and $\MaxSpec(\bfR_{\rF_i}^{\cC^\circ,\preceq})$, for $\rF_i$ varying over $1$-morphisms with local endomorphism rings, respectively.
Applying Corollary \ref{bijforindec} further provides a bijection between these two sets. This proves the claim.
\end{proof}

The following example shows that, even in the absence of local endomorphism rings, in some cases, cell $2$-representations are equivalent to those of the dg idempotent completion.
\begin{example}
Let $A$ be the path algebra of the quiver 
$$
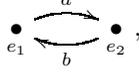
\xymatrix{
{\underset{e_1}{\bullet}}
\ar@/^0.5pc/[r]^a & 
{\underset{e_2}{\bullet}}
\ar@/^0.5pc/[l]^b 
},$$ 
viewed as a dg algebra with zero differential. Even though this algebra is infinite-dimensional, we can define $\cC_A$ as in Section \ref{CAdefsec}.
Denote by $\cC$ the dg sub-$2$-category whose $1$-morphisms are isomorphic to tensoring with those bimodules in $\ov{\{A, A\otimes_\Bbbk A \}}$. Then $\cC^\circ = \cC_A$ by construction.

It is easy to see that $\cC$ has a unique maximal weak two-sided cell $\J_{\tzero}$ given by $\ov{\{A\otimes_\Bbbk A \}}$, which is both a left and a right cell. Moreover, there is a bijection
$$\MaxSpec(\bfR^{\preceq,\cC}_{\J_{\tzero}}) \stackrel{\sim}{\longleftrightarrow} \{Ae_jA \,\vert \, j=1,2, \} \cup \{I_\lambda \, \vert \,\lambda \in \Bbbk^\times\}$$
where $I_\lambda = A(e_1-\lambda ba)A +A( e_2-\lambda ab)A$ for $\lambda \in \Bbbk^\times$ (see \cite[Exercise III.13]{ASS}) which is the annihilator of a two-dimensional simple $A$-module. Explicitly, the ideal $\bfI$ in  $\bfR^{\preceq,\cC}_{\J_{\tzero}}$ corresponding to $I\in \{Ae_jA \,\vert \, j=1,2, \} \cup \{I_\lambda \, \vert \,\lambda \in \Bbbk^\times\}$  is determined by $$\Hom_{\bfI(\bullet)}(A\otimes_\Bbbk A, A\otimes_\Bbbk A) \cong A\otimes_\Bbbk I.$$

Let now $\bfI$ be the ideal corresponding to some $\lambda \in \Bbbk^\times$ and consider $(\bfC_\bfI^{\preceq})^\circ$ as a pretriangulated $2$-representation of $\cC^\circ = \cC_A$. Put $X=A\otimes_\Bbbk A$ and $e=e_1\otimes e_1$. Then $X_e \cong Ae_1\otimes_\Bbbk e_1 A$ and $e_1\otimes e_1$ is not in $\bfI$ (since $e_1\notin I_\lambda$). Then $\bfR_{X_e}^{\preceq, \cC^\circ}$ is given by
the thick closure of $A\otimes_\Bbbk e_1 A$ and the ideal $\bfI_e$ corresponds to the maximal ideal of $e_1Ae_1$ generated by $e_1-\lambda ba$. However, $(\bfR_{X}^{\preceq})^\circ$ is given by the thick closure of $A\otimes_\Bbbk A$ inside of $\cC_A$.

Since $e$ is not in $\bfI$, the inclusion of $\bfC_{\bfI_e}^\preceq$ into $(\bfC_{\bfI}^\preceq)^\circ$ is fully faithful. 
Further observe that 
$1\otimes a \colon A\otimes e_1A \to A\otimes e_2A $ is an isomorphism in $(\bfC_{\bfI}^\preceq)^\circ$ with inverse $1\otimes \lambda b$. Thus, $\bfC_{\bfI_e}^\preceq=(\bfC_{\bfI_e}^\preceq)^\circ$ and $(\bfC_{\bfI}^\preceq)^\circ$ are dg equivalent. 
\end{example}

The following corollary shows that, under a mild technical condition, a weak two-sided cell is the apex of all  cell $2$-representations associated to its left cells.

\begin{corollary}\label{biggercellskill}
Let $\L$ and $\bfI$ be as in the statement of Lemma \ref {biggercellskill1} and let $\J$ be the weak two-sided cell containing $\L$. Further assume that not every product $\rX\rY$ for $\rX,\rY\in \J$ is strictly greater than $\J$ in the weak two-sided order. Then the weak apex of $\bfC_\bfI^\preceq$ is $\J$.
\end{corollary}

\proof
Let $\tilde \J$ be the weak apex of $\bfC_\bfI^\preceq$.
The condition on $\J$ ensures that $\J$ is not annihilated by $\bfC_\bfI^\preceq$, and hence $\J \preceq_J \tilde\J$. We thus need to show that every two-sided cell $\J'$ strictly greater than $\J$ in the order $\preceq_J$ is annihilated by $\bfC_\bfI^\preceq$.

Let $X\in \L$ be chosen such the endomorphism ring of $X$ is local, which is possible by Lemma \ref{biggercellskill1}. Further, let $Y\in \J'$, so $Y\succ_J X$. In particular,  $\bfG_{\bfR_\L^\preceq}(Y)$ is a strict dg $2$-subrepresentation of $\bfR_\L^\preceq$. Let
$\bfI'$ be the dg ideal in $\bfR_\L^\preceq$ generated by $\bfI$ and $\id_Y$, and assume $\bfI'$ strictly contains $\bfI$. We claim that $\bfI'$ does not equal $\bfR_\L^\preceq$ and hence contradicts maximality of $\bfI$. For a contradiction assume that $\bfI'$ equals $\bfR_\L^\preceq$. This implies we can write $\id_X = a\circ \id_{\rG Y}\circ b + g$ for $g$ in the dg ideal $\bfI$.

Notice that $g$ being in $\bfI$ implies that $g$ is not a unit, and hence contained in the unique maximal ideal of $\End(X)$. Thus $\id_X-g$ is invertible, so $a\circ \id_{\rG Y}\circ b$ is invertible and $X\succeq_L Y$, which is a contradiction. Hence $\bfI=\bfI'$ and $Y$ is annihilated by $\bfC_\bfI^\preceq$.
\endproof

The weak cell combinatorics of a dg $2$-category are linked to the cell combinatorics of the underlying additive $\Bbbk$-linear category $[\cC]$ which is obtained from $\cC$ by forgetting the differential and grading on the spaces of $2$-morphisms. We denote by $[\rF], [\alpha]$ the image in $[\cC]$ of a $1$-morphism $\rF$, respectively, a $2$-morphism $\alpha$ in $\cC$. 

\begin{lemma}\label{weakadditivecells}
Let $\rF=\left(\bigoplus_{i=1}^s \rF_i, \alpha\right),\rG=\left(\bigoplus_{j=1}^t \rG_j, \beta\right)$ by $1$-morphisms in $\cC$, where all $\rF_i$ and $\rG_j$ have local endomorphism rings. Then $\rF\preceq_L \rG$ if and only if for any $j=1,\ldots,t$ there exist an $i=1,\ldots,s$ such that $[\rF_i]\leq_L [\rG_j]$.

In particular, if $\rF$, $\rG$ have local endomorphism rings as $\Bbbk$-algebras, then $\rF\preceq_L\rG$ if and only if $[\rF]\leq_L[\rG]$.

The same statements hold when replacing $L$ by $R$ or $J$ in the respective partial orders.
\end{lemma}

\proof
Assume that for any $j=1,\ldots,t$ there exist an $i=1,\ldots,s$ such that $[\rF_i]\leq_L [\rG_j]$. 
This means that there exist $1$-morphisms $\rH_{ji}$ and $2$-morphisms $p_{ji}, \iota_{ji}$ such that $\id_{\rG_j}=p_{ji}\circ \id_{\rH_{ji}\rF_i} \circ \iota_{ji}$. By Lemma \ref{Rdifference}, this implies $\rF_i\preceq_L \rG_j$. 
Again by Lemma \ref{Rdifference}, $\rF\preceq_L \rF_k$ for every $k$. Hence $\rF\preceq_L \rG_j$ for any $j$, and $\id_{G_j}\in \bfI_{\bfP_\ti}(\id_\rF)$. Thus $\id_\rG=\bigoplus_{j=1}^t \id_{\rG_j}$ is in  $\bfI_{\bfP_\ti}(\id_\rF)$, as required.

Conversely, assume $\rF\preceq_L \rG$. Fix $\rG_j$. Since $\rG\preceq_L \rG_k$ for every $k$, we know that $\rF\preceq_L \rG_j$. Hence, by Lemma \ref{Rdifference}, we have $\id_{\rG_j}=p\circ \id_{\rH\rF} \circ \iota$ for some $1$-morphism $\rH$ and $2$-morphisms $p, \iota$.   Notice that $\id_{\rG_j}=p\circ \id_{\rH\rF} \circ \iota = \bigoplus_{k=1}^s p\circ \id_{\rH\rF_k} \circ \iota$. 
Since $\rG_j$ has a local endomorphism ring, one $p\circ \id_{\rH\rF_i} \circ \iota$ has to be an isomorphism. Hence for that $\rF_i$, we have  $\rF_i\preceq_L \rG_j$, and moreover, $[\rF_i]\leq_L [\rG_j]$.
\endproof

\begin{remark}
Observe that  the first part of the proof does not use locality of endomorphism rings. Therefore, for $1$-morphisms $\rF=\left(\bigoplus_{i=1}^s \rF_i, \alpha\right)$ and $\rG=\left(\bigoplus_{j=1}^t \rG_j, \beta\right)$, if for any $j=1,\ldots,t$ there exist an $i=1,\ldots,s$ such that $\rF_i\preceq_L \rG_j$, then $\rF\preceq_L \rG$.
\end{remark}

\subsection{Comparison to (non-dg) finitary \texorpdfstring{$2$}{2}-representation theory}\label{secfinitarycomp}

Let $\cD$ be a finitary $2$-category in the sense of \cite{MM1}. Then $\cD$ is a dg $2$-category concentrated in degree zero with the zero differential. Let $\cC=\ov{\cD}$ be its associated pretriangulated $2$-category. In particular, $\cC(\ti,\tj)$ is dg equivalent to the dg category of bounded complexes over $\cD(\ti,\tj)$ for any pair of objects $\ti,\tj$.

\begin{proposition}
The equivalence classes of (weak) dg cell $2$-representations of $\cC$ are in bijection with the equivalence classes of cell $2$-representations of $\cD$. 
\end{proposition}

\proof
Let $\bfC_\bfI^\preceq$ be a weak dg cell $2$-representation associated to a weak left cell $\L$ of $\cC$ and $\bfI\in \MaxSpec(\bfR^{\preceq}_\L)$. As $\cC$ satisfies the assumptions from Section~\ref{locendosec}, by Lemma \ref{biggercellskill1}, we can assume without loss of generality that $\L$ contains an indecomposable $1$-morphism of $\cD$. Let $\rF_1,\dots, \rF_s$ be the indecomposable $1$-morphisms of $\cD$ contained in $\L$. It is easy to see that $\L_{\cD}:=\{\rF_1,\dots, \rF_s\}$ forms a left cell in $\cD$. Indeed, the $\rF_i$ are equivalent in $\cD$ and thus equivalent in $\cC$. Conversely, every component of a $1$-morphism $\cC$-generated by $\L$ is either equivalent to or bigger than the $\rF_i$ when considered as a $1$-morphism in $\cD$.

By locality of the endomorphism rings of $\rF_1, \dots, \rF_s$ implies that $\id_\rG$ is contained in $\bfI$ for any indecomposable $\rG$ strictly left greater than $\L$ in the weak order. The ideal $\bfI$ is thus uniquely determined by $\End_{\bfI}(\rF_1\oplus \cdots \oplus \rF_s)$. The finitary $2$-subrepresentation of $\cD$ of $\bfR_\L^\preceq$ generated by $\rF_1\oplus \cdots \oplus \rF_s$ has a unique maximal ideal \cite[Section~6.2]{MM2}, and hence $\bfI$ is uniquely determined as well. Thus we have a unique (weak) dg cell $2$-representation associated to $\L$, which bijectively corresponds to the unique cell $2$-representation of the left cell $\L_{\cD}$ of $\cD$ and is, in fact, equivalent to its pretriangulated hull.
\endproof

\begin{remark}
For multitensor (or multiring) categories in the sense of \cite{EGNO}, an analogue of the left order used for finitary $2$-categories can be defined. Given that multitensor categories are Krull--Schmidt, each left cell has a unique cell $2$-representation. Considering the associated pretriangulated $2$-category as above, an analogue of the above proposition thus holds in this case.
\end{remark}

\subsection{Triangulated cell structure}\label{sectriangcell}

Consider the homotopy $2$-representation $\bfK\bfP_\ti$ associated to the principal pretriangulated $2$-representation $\bfP_\ti$. 

\begin{definition}\label{def:tricells} Let $\rF\in \cC(\ti,\tj)$ and $\rG\in \cC(\ti,\tk)$.
We say $\rF$ is {\bf left triangulated less than or equal to} $\rG$ if $\bfG_{\bfK\bfP_\ti}(\rG)\subseteq \bfG_{\bfK\bfP_\ti}(\rF)$. If this is the case, we write $\rF\leq_{L}^\Delta\rG$.
\end{definition}

Recall that $\mathcal{S}(\cC)$ is the set of dg isomorphism classes of dg indecomposable $1$-morphisms in $\cC$ up to shift.
We say that $\rF$ and $\rG$  in $\mathcal{S}(\cC)$ are {\bf triangulated left equivalent}, and write $\rF\steq_L^\Delta\rG$, if both $\rF\leq_{L}^\Delta\rG$ and $\rG\leq_{L}^\Delta\rF$. The corresponding equivalence classes on $\mathcal{S}(\cC)$ are called {\bf  triangulated left cells}.

The following lemma is immediate from the definitions.

\begin{lemma}\label{leqthentrileq}
If $\rF\leq_{L}\rG$ then $\rF\leq_{L}^\Delta\rG$.
\end{lemma}

The above lemma implies that any triangulated left cell is a union of strong left cells.

\begin{remark}
Given the involvement of $2$-morphisms not annihilated by the differential in the definition of the weak order, an analogue of such a weak order using the homotopy $2$-representations does not make sense.
\end{remark}

The following gives a converse to Lemma \ref{leqthentrileq} under additional technical conditions.

\begin{proposition}
Assume that $\cZ(\cC)$ is locally Krull--Schmidt. Then for two non-acyclic dg indecomposable $1$-morphisms $\rF, \rG$,
$$\rF\leq_L^\nabla\rG \qquad \text{ if and only if } \qquad \rF\leq_L\rG.$$ 
\end{proposition}
\begin{proof}
Let $\rF\in \cC(\ti,\tj)$, $\rG\in \cC(\ti,\tk)$.
By Lemma \ref{leqthentrileq}, $\rF\leq_L \rG$ implies $\rF\leq_L^\nabla\rG$.

To prove the converse, assume that $\rF\leq_L^\nabla\rG$. Then there exists a $1$-morphism $\rH\in\cC(\tj,\tk)$ such that $\rG$ is a direct summand of $\rH\rF$ in the homotopy category $\K(\cC(\ti,\tk))$. This implies that there exist dg morphisms $p\colon\rH\rF\to \rG$, $\iota\colon \rG\to \rH\rF$ such that $p\iota=\id_\rG+\del(f)$ for some endomorphism $f$ of $\rG$.

Since the $1$-morphism $\rG$ is dg indecomposable and $\Z(\cC(\ti,\tk))$ is Krull--Schmidt, the ring $\End_{\Z(\cC(\ti,\tk))}(\rG)$ of its dg endomorphisms
is local. As $\rG$ is not acyclic, $\del(f)$ cannot be a unit and hence is in the radical. Indeed, if $\del(f)$ is invertible with inverse $g$, a dg morphism, then 
$\id_{\rG}=\del(f)g=\del(fg)$, which contradicts $\rG$ being non-acyclic.

Thus, $(\id+\del(f))$ is invertible and 
$$p\iota(\id_\rG+\del(f))^{-1}=\id_\rG.$$
Thus, setting $\iota'=(\id_\rG+\del(f))^{-1}$, $\iota'p$ is a dg idempotent displaying $\rG$ as a dg direct summand of $\rH\rF$ in $\Z(\cC(\ti,\tk))$. Thus, $\rF\leq_L\rG$.
\end{proof}

\begin{corollary}\label{strongequalstriang}
If $\cZ(\cC)$ is locally Krull--Schmidt, then non-acyclic strong left cells and non-acyclic triangulated left cells coincide.
\end{corollary}

\section{The dg \texorpdfstring{$2$}{2}-category \texorpdfstring{$\cC_A$}{CA}}\label{CAsec}

\subsection{Definition of \texorpdfstring{$\cC_A$}{CA}}\label{CAdefsec}

Let $A=A_\tone\times\cdots \times A_\tn$ be a finite-dimensional dg algebra, where each $A_\ti$ is indecomposable as a $\Bbbk$-algebra. We further assume that we have a decomposition of $1_A = e_1+\cdots +e_s$ into primitive orthogonal idempotents such that each $e_i$ is annihilated by the differential. 

We define an equivalence relation on $\{e_1, \cdots, e_s\}$ by saying $e_i$ is equivalent to $e_j$ if there exist $a,b\in A$ with $ab=e_i$ and $ba=e_j$ and, moreover, $\del(a)=\del(b)=0$.
We let $\tE$ denote a set of representatives of equivalence classes in $\{e_1, \cdots, e_s\}$.

Let $\A_\ti$ be a small dg category dg equivalent to $\widehat{\{A_\ti\}}$ inside $A_\ti\dgmod$, and set $\A = \coprod_{\ti=\tone}^\tn \A_\ti$. Note that, by definition, $\A$ is a dg idempotent complete pretriangulated category.

\begin{definition}
We define $\cC_A$  as the dg $2$-category with 
\begin{itemize}
\item objects $\mathtt{1}, \dots, \mathtt{n}$ where we identify $\mathtt{i}$ with $\A_\ti$;
\item $1$-morphisms in $\cC_A(\ti,\tj)$ are functors from $\A_\ti$ to $\A_\tj$ dg isomorphic to tensoring with dg $A_\tj$-$A_\ti$-bimodules in the thick closure of $A_\tj\otimes_\Bbbk A_\ti$ and, additionally, $A_\ti$ if $\ti=\tj$;
\item $2$-morphisms all natural transformations of such functors.
\end{itemize}
\end{definition}

Observe that, by definition, $\cC_A$ satisfies the assumptions of Section \ref{locendosec}.

For future use, we also define the notation of $\rF_{\tj,\ti}$ for the functor given by tensoring with $A_\tj\otimes_\Bbbk A_\ti$ and $\rF_{e_k,e_l}$ for the functor given by tensoring with $A e_k\otimes_\Bbbk e_l A $.

We define the {\bf natural $2$-representation} $\bfN$ of $\cC_A$ as its defining action on $\A$, that is,  $\bfN(\ti) = \A_\ti$, $\bfN(\rF) = \rF$ for a $1$-morphism $\rF$, and $\bfN(\alpha) = \alpha$ for a $2$-morphism $\alpha$. By construction, $\cC_A$ is a pretriangulated $2$-category, and $\bfN$ is in $\cC_A\tworep$. It is cyclic with any $A_\ti\in \bfN(\ti)=\A_\ti$ as a $\cC$-generator.

\subsection{Weak cell structure of \texorpdfstring{$\cC_A$}{CA}}\label{secweakCA}

Set $$\tI =\{\ti\;\vert \;\tone\leq \ti\leq  \tn,  A_\ti \text{ is not semisimple as a $\Bbbk$-algebra}\}.$$

\begin{proposition}\label{weakcells}
The $2$-category $\cC_A$ has 
\begin{itemize}
\item a weak two-sided cell $\J_\tzero$ consisting of all dg indecomposable $1$-morphisms in the thick closure of $\{\rF_{\tj,\ti} \;\vert \;\tone\leq \ti,\tj\leq  \tn \}$ where $\rF_{\tj,\ti}$ is given by tensoring with $A_\tj\otimes_\Bbbk A_\ti$;
\item for each $\ti\in \tI$, a weak two-sided cell $\J_\ti$ containing $\one_\ti$, consisting of all dg indecomposable $1$-morphisms $\rG$ which admit $2$-morphisms $p\colon \rG\to \one_\ti$ and $\iota\colon \one_\ti\to\rG$ such that $p\iota = \id_{\one_\ti}$.
\end{itemize}
This yields an exhaustive list of pairwise distinct weak two-sided cells. 
\end{proposition}

\proof
For $\ti\in \tI$, let $\J_\ti$ denote the weak two-sided cell containing $\one_\ti$ and $\J_\ti'$ the set of dg indecomposable $1$-morphisms $\rG$ is such that there exist $2$-morphisms $p\colon \rG\to \one_\ti$ and $\iota\colon \one_\ti\to\rG$ such that $p\iota = \id_{\one_\ti}$. By definition, $\rG\preceq_J \one_\ti$ for any $\rG\in \J_\ti'$. Moreover, this condition implies that $\rG\in \cC(\ti,\ti)$, so since $\id_\rG= \id_{\rG\one_\ti} = \id_{\one_\ti\rG}$, we further see that $\rG\succeq_J \one_\ti$. Hence, $\J_\ti'\subseteq \J_\ti$.

To prove $\J_\ti\subseteq \J_\ti'$, let $M'\in \ov{\{A_\ti,A_\ti\otimes A_\ti \}}$ and assume $M$ is a dg direct summand of $M'$, which is not projective as an $A$-$A$-bimodule. We need to show that $M$ defines a $1$-morphism in $\J_\ti'$. Indeed, assume $M$ is a dg direct summand of $M'$ defined by dg morphisms $p',\iota'$.
Then $M'$ cannot be a projective $A$-$A$-bimodule either so there exist filtration subquotients $M'_1,\ldots,M'_t$ of $M'$ which are dg isomorphic to $A_\ti$.
Let $e\colon M'\to M'$ be the, not necessarily dg, idempotent projecting onto the sum of $M'_1,\ldots,M'_t$. We claim that $p'\circ e\circ \iota'\notin \rad \End_{A-A}(M)$. Indeed, if it were, then 
$$\id_M-p'\circ e\circ \iota' = p'\circ\iota' -p'\circ e\circ \iota' = p'\circ( \id_{M'}-e)\circ \iota'$$
would be invertible and $M$ would be projective as an $A$-$A$-bimodule.

Since $p'\circ e\circ \iota'$ factors over $M'_1\oplus\ldots\oplus M'_t$, there exists $i\in \{1,\ldots, t\}$ such that the component maps $M'_i\to M$, and $M \to M'_i$ are a split monomorphism, respectively, a split epimorphism \cite[Section V.7, Exercise 7.(c)]{ARS}.

Now assume $\rG\in \J_\ti$. Let $M$ be a dg bimodule such that $\rG$ is isomorphic to $M\otimes_A -$. 
Being in $\J_\ti$, in particular means that $\rG\preceq_J \one_\ti$, i.e.\ that $\id_{\one_\ti}=p\circ\id_{\rH_1\rG\rH_2}\circ\iota$ for some $1$-morphisms $\rH_1,\rH_2$ and $2$-morphisms $\iota,p$. This implies that $M$ cannot be projective as an $A$-$A$-bimodule, and in particular, $\rG\in \cC(\ti,\ti)$, since for $\tk\neq \tj$, all $1$-morphisms in $\cC(\tj,\tk)$ are isomorphic to tensoring with projective $A$-$A$-bimodules. The split monomorphism and split epimorphism from the previous paragraph now yield $p\colon \rG\to \one_\ti$ and $\iota\colon \one_\ti\to \rG$ as claimed. Hence, $\J_\ti\subseteq \J_\ti'$ and the two sets are equal.

In particular, if $\tj\in \tI$ and $\tj\neq \ti$,  we have that $\one_\tj\not\preceq_J \one_\ti$ since $\one_\tj\notin \cC(\ti,\ti)$. Thus, the weak cells $\J_\ti$ and $\J_\tj$ are disjoint for $\ti,\tj\in \tI$ such that $\ti\neq \tj$. Hence we have seen that the union of all $\J_\ti$, for $\ti\in \tI$, contains all dg indecomposable $1$-morphisms that are dg isomorphic to tensoring with a non-projective $A$-$A$-bimodule.

Let $\J_\tzero$ denote the set consisting of all dg indecomposable $1$-morphisms in in the thick closure of $\{\rF_{\tj,\ti} \;\vert \;\tone\leq \ti,\tj\leq  \tn \}$. 
It remains to show that $\J_\tzero$ is a single weak two-sided cell containing all dg indecomposable $1$-morphisms $\rH$ not in any of the $\J_\ti$.

It follows from Lemma \ref{weakadditivecells} that all $\rF_{e_k,e_l}$ are contained in the same weak two-sided cell, which we denote by $\J$. Further, $\rF_{e_k,e_l}\not\preceq_J \one_\ti$ for $\ti\in \tI$. This follows, arguing similarly to above, using the fact that $\one_\ti$ does not correspond to tensoring with a projective $A$-$A$-bimodule. Thus $\J$ is disjoint from all $\J_\ti$, for $\ti\in \tI$. 

Moreover, $\bigoplus_{\ti,\tj\in \tI}\rF_{\tj,\ti}$ is dg isomorphic to the dg direct sum $\bigoplus_{k,l=1}^s \rF_{e_k,e_l}$. It follows that in the thick closure of $\{\rF_{\tj,\ti} \;\vert \;\tone\leq \ti,\tj\leq  \tn \}$ is equal to the thick closure of $\{\rF_{e_k,e_l}\,\vert\,k,l=1,\dots,s\}$. 

We first prove that $\J_\tzero\subseteq \J$. Let $\rG\in \J_\tzero$ and choose $\rG'$ in $\ov{\{\rF_{e_k,e_l}\,\vert\,k,l=1,\dots,s\}}$ such that $\rG$ is a nonzero dg direct summand of $\rG'$. Then there exist morphisms $p\colon \rG\to \rG',\iota\colon \rG'\to \rG$ such that $\id_{\rG}=p\circ \id_{\rG'}\circ \iota$. Note that there exists a filtration subquotient $\rF_{e_r,e_q}$ of $\rG'$, given by (not necessarily dg) morphisms $p'\colon \rG'\to \rF_{e_r,e_q}$ and $\iota'\colon \rF_{e_r,e_q}\to \rG'$ such that the composition
$p\circ\iota'\circ p'\circ\iota$ is not a radical morphism. Indeed, otherwise $\id_\rG$ would be a radical morphism forcing $\rG$ to be zero. Thus, $p'\circ \iota\circ p \circ \iota'$ is also not a radical morphism and, as $\rF_{e_r,e_q}$ has a local endomorphism ring, we conclude that $\rF_{e_r,e_q}\succeq_J \rG$. Since $\rG$ is in the thick closure of the $\rF_{e_k,e_l}$, which are all equivalent, we see that $\rG\succeq_J \rF_{e_r,e_q}$. Hence, we have shown that $\J_\tzero\subseteq \J$.

We have seen that the set of all dg indecomposable $1$-morphism of $\cC_A$ is the disjoint union of $\J_\tzero$ and the $\J_\ti$ for $\ti\in \tI$. Hence, $\J\subseteq \J_\tzero$ and $\J=\J_\tzero$. This completes the proof.
\endproof

By Lemma \ref{biggercellskill1}, in order to classify weak dg cell $2$-representations, it suffices to consider weak left  cells containing a $1$-morphism isomorphic to the identity or tensoring with $Ae_i\otimes e_jA$. The next corollary provides a complete list of these weak left cells, as well as their analogous weak right cells. Recall that $\tE$ denotes a set of representatives of equivalence classes in $\{e_1, \cdots, e_s\}$, defined at the beginning of this section.

\begin{corollary}\label{weakleft}
The following lists some distinct weak left and right cells of $\cC_A$.
\begin{enumerate}[(a)]
\item\label{weaklefta} The $\J_\ti$, for $\ti\in \tI$, are both weak left and right cells. 
\item\label{weakleftb} The weak two-sided cell $\J_\tzero$ contains distinct weak left cells $\L_{\tzero,e}$ for $e\in \tE$,  where $\L_{\tzero,e}$ contains $\{\rF_{e_k,e}\vert e_k\in \tE\}$.
\item\label{weakleftc} 
Similarly, the weak two-sided cell $\J_\tzero$ contains distinct weak right cells $\R_{\tzero,e}$ for $e\in \tE$ where $\R_{\tzero,e}$ contains $\{\rF_{e,e_k}\vert e_k\in \tE\}$.
\end{enumerate}
\end{corollary}

\proof
Part~(a) is clear from the first paragraph of the proof of Proposition \ref{weakcells}.

By Lemma \ref{weakadditivecells}, all $1$-morphisms dg isomorphic to tensoring with  $Ae'\otimes eA$ for any idempotent $e'$ are in the same weak left cell, which we denote by $\L_{\tzero,e}$. Further, the weak left cells $\L_{\tzero,e}$ are disjoint. This proves Part (b) and Part (c) is proved analogously.
\endproof

\subsection{Classification of weak cell \texorpdfstring{$2$}{2}-representations of \texorpdfstring{$\cC_A$}{CA}}

To classify all weak cell $2$-representations, we first consider those with weak apex $\J_\tzero$. By Lemma \ref{biggercellskill1}, we can restrict to the dg cell $2$-representations associated to the weak left cells given in Corollary \ref{weakleft}.

\begin{proposition}\label{uniquecell2rep0}
For $e\in \tE$, there is a unique weak cell $2$-representation $\bfC^\preceq_{\L_{\tzero,e}}$ associated to $\L_{\tzero,e}$.
This cell $2$-representation is dg equivalent to $\bfN$ if $\del(\rad eAe)\subseteq \rad eAe$, and is acyclic otherwise.
\end{proposition}
\proof
Set $\L=\L_{\tzero,e}$ and $\bfR=\bfR_{\L_{\tzero,e}}^\preceq$.

Let $\bfI\in \Spec(\bfR)$. 
By Lemma \ref{easyideals}\eqref{idealonoverline} and Corollary \ref{weakleft}, morphism spaces in each $\bfI(\ti)$ are completely determined by the subspaces $\Hom_{\bfI(\ti)}(X,Y)$, for $X,Y\in \overline{\{\rF_{e',e} \vert e'\in \tE\}}$.
Composing with injections and projections,  $\bfI$ is therefore uniquely determined by the subspaces $\Hom_{\bfI(\ti)}(\rF_{e',e}, \rF_{e'',e})$ of $$\Hom_{\bfR(\ti)}(\rF_{e',e}, \rF_{e'',e})\cong e'Ae'' \otimes eAe.$$
We claim that under this dg isomorphism, $\Hom_{\bfI(\ti)}(\rF_{e',e}, \rF_{e'',e})$ is contained in $e'Ae''\otimes \rad eAe$. Indeed, applying $\rF_{e',e'}$ to a morphism not in $e'Ae''\otimes \rad eAe$ and using appropriate injections and projections, we obtain $\id_{\rF_{e',e}}$ in $\bfI(\ti)$.

Let $\rad_\del eAe$ be the sum of all proper ideals in $eAe$ which are closed under the differential. Then this is the unique maximal dg ideal in $eAe$. Since the morphism space $\Hom_{\bfI(\ti)}(\rF_{e',e}, \rF_{e'',e})$ needs to be closed under $\del$, it in fact needs to correspond to a subspace of $e'Ae''\otimes \rad_\del eAe$.

Consider $\bfI_{\L}=\sum_{\bfI \in \Spec(\bfR)}\bfI$. This is an ideal, and any 
$\Hom_{\bfI_{\L}(\ti)}(\rF_{e',e}, \rF_{e'',e})$ still corresponds to a subspace in $e'Ae''\otimes \rad_\del eAe$, hence $\bfI_\L$ is proper. It is thus the unique maximal dg ideal of $\bfR$ by construction. Since the ideal $\bfI$ given by 
$\Hom_{\bfI(\ti)}(\rF_{e',e}, \rF_{e'',e}) \cong e'Ae''\otimes \rad_\del eAe$ is a dg ideal and is $\cC$-stable by construction, it is maximal by the above argument and hence equals $\bfI_\L$.

Assuming $\del(\rad eAe)\subset \rad eAe$, we have $\rad_\del eAe=\rad eAe$, and
the statement that in this case $\bfC^\preceq_\L$ is equivalent to $\bfN$ is proved analogously to \cite[Theorem 6.12]{LM} (see also \cite[Proposition 9]{MM5}).

Suppose $\del(\rad eAe)\not\subset \rad eAe$. Since the codimension of $\rad eAe$ in $eAe$ is one, there exists $x\in \rad eAe$ such that $e=\del(x)$. For any $e'\in\tE$, the identity on the image of $\rF_{e',e}$ in $\bfC^\preceq_\L(\ti)$ is then given by $\del(e'\otimes x) + e'Ae'\otimes \rad_\del eAe$. Hence $\bfC^\preceq_\L$ is acyclic in this case.
\endproof

\begin{corollary}
Let $e,e'\in \tE$ such that $\del(\rad eAe)\subset \rad eAe$ and $\del(\rad e'Ae')\subset \rad e'Ae'$. Then $\bfC_{\L_{\tzero,e}}^\preceq$ and $\bfC_{\L_{\tzero,e'}}^\preceq$ are dg equivalent. In particular, any weak cell $2$-representation with weak apex $\J_\tzero$ is either dg equivalent to $\bfN$ or acyclic.
\end{corollary}

Now we turn to the weak cell $2$-representations with weak apex $\J_i$. For this, given a dg algebra $A$, we denote by $Z(A)$ its centre and observe that the latter is closed under the differential.

\begin{proposition}\label{uniquecell2repi}
For $\ti \in \tI$, there is a unique cell $2$-representation $\bfC^\preceq_\ti$ associated to the two-sided cell $\J_\ti$, which is also a left cell. The cell $2$-representation $\bfC^\preceq_\ti$ is not acyclic if and only if  $\rad_\del Z(A_\ti)  = \rad Z(A_\ti)$.

More specifically,  $\bfC^\preceq_\ti(\tj)=0$ for $\tj\neq \ti$. Moreover, if $\rad_\del Z(A_\ti)  = \rad Z(A_\ti)$, then $\bfC^\preceq_\ti(\ti)$ is dg equivalent to $\Bbbk\dgmod$.
\end{proposition}

\proof
By Corollary \ref{weakleft}\eqref{weaklefta}, $\J_\ti$ is indeed a left cell, which we denote by $\L_\ti$. Let $\bfI\in \MaxSpec(\bfR_{\L_\ti}^{\preceq})$. The weak left cells $\L_{\tzero,e}$ are strictly greater than $\L_\ti$ if and only if $e\in A_\ti$, and unrelated otherwise. Hence, for $e'\in A_{\tj}$,  $\rF_{e',e}\in \bfR_{\L_\ti}^\preceq(\tj)$ if and only if $e\in A_\ti$. By Corollary \ref{biggercellskill}, the weak apex of $\bfC^\preceq_\bfI$ is $\J_\ti$, which implies that $\id_{\rF_{e',e}}\in \bfI(\tj)$.

By Lemma \ref{2repidealonoverline}, the ideal $\bfI$ is determined by its intersections with morphism spaces involving the objects $\{\one_\ti, \rF_{e',e}\,\vert\, e'\in A, e\in A_\ti\}$. All morphisms with any of the $\rF_{e',e}$ as source or target are in the ideal $\bfI$. In particular $\bfC_\bfI^\preceq(\tj)=0$ for $\tj\neq \ti$. It hence suffices to determine $$\End_{\bfI(\ti)}(\one_\ti)\subseteq \End_{\bfR_{\L_i}^\preceq(\ti)}(\one_\ti)\cong \End_{A-A}(A_\ti)\cong Z(A_\ti).$$
The latter is a local ring, hence has a unique maximal dg ideal $\rad_\del Z(A_\ti)$, which under the above isomorphism contains $\End_{\bfI(\ti)}(\one_\ti)$. Maximality of $\bfI$ implies the other inclusion, which proves uniqueness of $\bfI$ and hence of the associated cell $2$-representation. We set $\bfC^\preceq_\ti = \bfC_\bfI^\preceq$.

If $\rad_\del Z(A_\ti)  \subsetneq \rad Z(A_\ti)$, similarly to the proof of Proposition \ref{uniquecell2rep0}, there exists an $x\in Z(A_\ti)$ such that $1_{A_\ti} = \del(x)$, so $\id_{\one_\ti}$ is in the image the differential, and hence all objects in $\bfC^\preceq_\ti(\ti)$ are acyclic. 

If $\rad_\del Z(A_\ti)  = \rad Z(A_\ti)$, it follows from the above considerations that $\bfC^\preceq_\ti(\ti)$ is dg equivalent to $\Bbbk\dgmod$.
\endproof

We summarise our results from Propositions \ref{uniquecell2rep0} and \ref{uniquecell2repi} in the following theorem.

\begin{theorem}\label{allweakcell2reps}
Each weak two-sided cell in $\cC_A$ has a unique associated weak cell $2$-representation, and this gives a complete list of weak cell $2$-representations.
\end{theorem}

\subsection{Classification of strong cell \texorpdfstring{$2$}{2}-representations of \texorpdfstring{$\cC_A$}{CA}}

By Theorem \ref{maxspecbij}, uniqueness of maximal ideals for weak left cells in $\cC_A$ implies that $\MaxSpec(\bfR_{\L}^{\leq})$ also has a unique element for any strong left cell $\L$. Hence, we can again denote the associated cell $2$-representation by $\bfC^\leq_\L$.

\begin{proposition}
Let $\L'$ be a weak left cell inside the weak two-sided cell $\J_0$ such that $\bfC^\preceq_{\L'}$ is not acyclic. Let $\L$ be a strong left cell inside $\L'$.
Then $\bfC^\leq_\L$ is either dg equivalent to  $\bfC^\preceq_{\L'}$ or acyclic.\end{proposition}

\proof
By assumption, $\bfC^\preceq_{\L'}$ is equivalent to $\bfN$ and, by Proposition \ref{cellinclude}, $\bfC^\leq_\L$ can be considered as a dg $2$-subrepresentation of $\bfN$.

If all $X\in \bfC^\leq_\L(\ti)$ are acyclic, then $\bfC^\leq_\L$ is acyclic. So assume
that there exists a dg indecomposable $X\in \bfC^\leq_\L(\ti)$ which is not acyclic. 
Consider $X$ as a dg $\Bbbk$-module by forgetting the $A$-module structure. Recall that our chosen set of idempotents is annihilated by the differential. Hence, we can choose a decomposition of $X$ into dg indecomposable $\Bbbk$-modules such that, for each dg indecomposable direct summand $X_i$, there exists some $e_i$ such that $X_i=e_iX_i$.  Non-acyclicity  of $X$ implies that there exists a $1$-dimensional direct summand $S$, so let $e$ be such that $S=eS$. Then $Ae\otimes eA\otimes_AX$ has a dg direct summand given by $Ae\otimes S$ which is dg isomorphic to $Ae$. For any $e'\in \tE$, we then obtain $Ae'$ as a dg direct summand of $Ae'\otimes eA\otimes_A Ae$ by choosing the dg direct summand $Ae'\otimes e$. Thus $\bfC^\leq_\L$ is equivalent to $\bfN$.
\endproof

\begin{proposition}
Let $\L$ be a strong left cell inside the weak two-sided cell $\J_\ti = \L_\ti$ for some $\ti\in \tI$ such that $\rad_\del Z(A_\ti) = \rad Z(A_\ti)$. 
Then $\bfC^\leq_\L$ is either dg equivalent to  $\bfC^\preceq_{\L_\ti}$ or acyclic.\end{proposition}

\proof
Note that under the assumption that $\rad_\del Z(A_\ti) = \rad Z(A_\ti)$, the proof of Proposition \ref{uniquecell2repi} shows that $\bfC^\preceq_{\L_{\ti}} (\ti)$ is dg equivalent to $\Bbbk\dgmod$ and $\bfC^\preceq_{\L_{\ti}} (\tj) =0$ for $\ti\neq \tj$.
The proof is then analogous to, but easier than, the proof of the previous proposition.
\endproof

\begin{corollary}\label{allstrongcell2reps}
Any strong cell $2$-representation of $\cC_A$ is either dg equivalent to a weak cell $2$-representation or acyclic. Hence, any non-acyclic strong cell $2$-representation of $\cC_A$ is either dg equivalent to the natural $2$-representation $\bfN$ or to $\bfC_\ti^{\preceq}$ as defined in Proposition~\ref{uniquecell2repi}.
\end{corollary}

\subsection{Explicit examples}

Assume first that $A$ is any finite-dimensional $\Bbbk$-algebra, i.e.\ a dg algebra with zero differential. Then $1$-morphisms in $\cC_A$ are given by bounded complexes of $A$-$A$-bimodules in the additive closure of $A\oplus A\otimes_\Bbbk A$. 
In particular, $\cC_A$ is the pretriangulated hull of the finitary $2$-category with the same name defined in \cite[Section 4.5]{MM3}, c.f.\ \cite[Section 5.2]{LM2}. For simplicity, assume $A$ is connected, i.e.\ $1$ is the only nonzero central idempotent. By Theorem \ref{allweakcell2reps} and Corollary \ref{allstrongcell2reps}, in this case, there are precisely two cell $2$-representations of $\cC_A$, neither of which is acyclic. 
\begin{itemize}
\item The first is given by the action on $\Bbbk\dgmod$, where projective bimodules act by zero, and the identity bimodule acts as the identity and any non-acyclic complex acts simply by the sum of (shifted) identities stemming from its identity bimodule components. 
\item The second is induced by the natural action of bounded complexes of the given $A$-$A$-bimodules on the category of bounded complexes of projective $A$-modules by tensor product.
\end{itemize}
Note that, in particular, all cell $2$-representations for $\cC_A$ are pretriangulated hulls of finitary cell $2$-representations.

One specific example of an algebra of interest in this context is the case is so-called zigzag algebra $Z=Z_n$ for a natural number $n$. Indeed, as explained \cite[Section 5.5]{LM2}, any pretriangulated $2$-representation of the pretriangulated dg category $\ov{\cB_n}$ naturally obtained from the categorification of the braid group in \cite{KS} extends to a pretriangulated $2$-representation of $\cC_Z$, since $\cC_Z$ is the dg idempotent completion of $\ov{\cB_n}$. We remark, however, that in contrast to $\cC_Z$, the dg $2$-category $\ov{\cB_n}$ only has one weak cell, namely that of the identity. By Corollary \ref{bijforindec}  and Theorem  \ref{allweakcell2reps}, the maximal spectrum of this cell is a singleton, giving rise to the trivial $2$-representation on $\Bbbk\dgmod$.  Hence by Proposition \ref{cellinclude}, the categorified braid group action of \cite{KS}, which is the restriction of the homotopy $2$-representation of the natural cell $2$-representation $\bfN$ of $\cC_Z$, cannot be constructed as (the homotopy $2$-representation of) a cell $2$-representation for $\ov{\cB_n}$.

Closely related to this, \cite{QS} constructs a $p$-differential on the Koszul dual $Z^!$ of the zigzag algebra, which, in the case of $p=2$ (and hence $\mathrm{char} \,\Bbbk =2$), gives rise to a dg algebra structure on $Z^!$, satisfying the assumptions of this section. Since none of the idempotents are in the image of the differential, Theorem \ref{allweakcell2reps} and Corollary \ref{allstrongcell2reps}, together with Propositions \ref{uniquecell2rep0} and \ref{uniquecell2repi}, imply that $\cC_{Z^!}$ (where we consider $Z^!$ as a dg algebra with this differential) again has two (non-acyclic) cell $2$-representations up to dg equivalence, the trivial one on $\Bbbk\dgmod$ and the natural $2$-representation $\bfN$. Indeed, the categorified Burau representation defined in \cite[Section~5.1]{QS} is contained in the homotopy $2$-representation of the natural cell $2$-representation of $\cC_{Z^!}$, in the sense of acting on the same category, just with fewer functors.

\end{document}